\documentclass[11pt]{article}
\usepackage{graphicx,wrapfig}
\graphicspath{figures/}
\usepackage{flafter}
\usepackage{leftidx}
\usepackage{mathrsfs}
\usepackage{amssymb,amsmath}%\eqref
\usepackage{bbding}
\usepackage{fancyhdr}%µ÷ÓÃÅÅÓ¡Ò³ÃŒÓëÒ³œÅµÄºê°ü
\usepackage{mathrsfs}
\usepackage{color}
\usepackage{stmaryrd}
\usepackage{amsfonts}
\usepackage{latexsym}
\usepackage{psfrag}
\usepackage{graphicx,subfigure}
\usepackage{fancyhdr,graphicx}
\usepackage{multicol}
\usepackage{dsfont}
\usepackage{bbm}
\usepackage{booktabs}
\usepackage[center]{caption2}
\usepackage{cite}
\usepackage{epstopdf}
\usepackage{diagbox}%»­Ð±Ïß±íÍ·
\usepackage{booktabs}%ŽŽœšÃ»ÓÐÊúÏß·ÖžîµÄ±ížñ
\usepackage [utf8]{inputenc}%·ÀÖ¹Ìáœ»Ê±ÔËÐÐ³öŽí
\usepackage{multirow}%ºÏ²¢¶àÐÐµ¥Ôªžñ
\usepackage{enumerate}%ÁÐ±íºê°ü
\usepackage{enumitem}
\usepackage{algpseudocode,algorithm,algorithmicx}% Éú³ÉÎ±ŽúÂëËã·šÍŒ
\usepackage{appendix}
\usepackage[section]{placeins}
\usepackage{float}
\allowdisplaybreaks
\renewcommand{\leq}{\leqslant}

\date{}
\newtheorem{theorem}{Theorem}[section]
\newtheorem{lemma}{Lemma}[section]
\newtheorem{remark}{Remark}[section]
\newtheorem{example}{Example}[section]

\numberwithin{equation}{section}%°ŽÕÕÕÂœÚ±àºÅ

\newcommand{\zd}{\,\mathrm{d}}
\newcommand{\diff}{\triangledown_{\tau}}

\newcommand{\abs}[1]{\left|#1\right|}

\newcommand{\brab}[1]{\big(#1\big)}
\newcommand{\braB}[1]{\Big(#1\Big)}

\newcommand{\myinner}[1]{\left\langle#1\right\rangle}
\newcommand{\myinnerb}[1]{\big\langle#1\big\rangle}

\newcommand{\mynorm}[1]{\left\|#1\right\|}
\newcommand{\mynormb}[1]{\big\|#1\big\|}

\newcommand{\zhur}[1]{{\color{black}#1}}

\begin{document}
\title{Energy stability and convergence of variable-step L1 scheme for the time fractional Swift-Hohenberg model}
\author{
Xuan Zhao\thanks{Corresponding author. School of
		Mathematics, Southeast University, Nanjing 210096, P. R. China (xuanzhao11@seu.edu.cn).}
\quad Ran Yang \thanks{School of
Mathematics, Southeast University, Nanjing 210096, P. R. China. (220191527@seu.edu.cn).}
\quad Ren-jun Qi \thanks{School of
	Mathematics, Southeast University, Nanjing 210096, P. R. China. (rjqi97@163.com).}
\quad Hong Sun\thanks{School of
	Mathematics, Southeast University, Nanjing 210096, P. R. China and Department of Mathematics and Physics, 
Nanjing Institute of Technology, Nanjing 211167, P. R. China (sunhongzhal@126.com).}
}

%%%%%%%%%%%%%%%%%%%%%%%%%%%%%%%%%%%%%%%%%%%%%%%%%%%%%%%%%%%%%%%%%%%%%%%%%%%%%%%%%%%%%%%%%%%
\date{}
\maketitle
\normalsize
\begin{abstract}
A fully implicit numerical scheme is established for solving the time fractional Swift-Hohenberg (TFSH) equation with a Caputo time derivative of order $\alpha\in(0,1)$. The variable-step L1 formula and the finite difference method are employed for the time and the space discretizations, respectively. The unique solvability of the numerical scheme is proved by the Brouwer fixed-point theorem. With the help of the discrete convolution form of L1 formula, the time-stepping scheme is shown to preserve a discrete energy dissipation law which is asymptotically compatible with the classic energy law as $\alpha\to1^-$. Furthermore, the $L^\infty$ norm boundedness of the discrete solution is obtained. Combining with the global consistency error analysis framework, the $L^2$ norm convergence order is shown rigorously. Several numerical examples are provided to illustrate the accuracy and the energy dissipation law of the proposed method. In particular, the adaptive time-stepping strategy is utilized to capture the multi-scale time behavior of the TFSH model efficiently.
	
%The variable-step L1 formula is applied for the time fractional Swift-Hohenberg (TFSH) equation. The unique solvability and the $L^2$ norm convergence of the numerical scheme are proved strictly. Moreover, it is shown that the time-stepping scheme preserves a discrete energy dissipation law. Several numerical examples are provided to illustrate the accuracy and the energy dissipation law. In particular, the adaptive time-stepping strategy is utilized to capture the multi-scale time behavior of the TFSH model efficiently.

  %\noindent{\bf AMS subject classifications.}\;\; 35Q99, 65M06, 65M12, 74A50
\end{abstract}

\section{Introduction}\setcounter{equation}{0}
%%%%%%%%%%%%%%%%%%%%%%%%%%%%%%%%%%%%%%%%%%%%%%%%%%%%%%%%%%%%%%%%%%%%%
The classic Swift-Hohenberg (SH) equation, first derived by Jack Swift and Hohenberg through the study of the thermal convection of the Rayleigh-B$\acute{\rm e}$nard instability \cite{SH}, has extensive applications in the modeling and simulation of the pattern formations \cite{Kud2012,Rosa2000,Ibbeken2019,leesym,Lega}. As an important phase field model, the SH equation \cite{leecmame} is viewed as a gradient flow with the following Lyapunov energy functional  
\begin{align}\label{def:continuous energy}
E[u]:=\int_{\Omega}\frac{1}{2}u(1+\Delta)^2 u+F(u) \zd\mathbf{x},
\end{align}
i.e., $\partial_t u:=-\frac{\delta E}{\delta u}$ where $\frac{\delta }{\delta u}$ denotes the variational derivative, $u$ represents the density field and the nonlinear potential
\begin{align*}
	F(u)=\frac{1}{4}u^4-\frac{\mathrm{g}}{3}u^3-\frac{\epsilon}{2}u^2
\end{align*} 
with two physical parameters $\mathrm{g}\ge 0$ and $\epsilon>0$. Under the periodic conditions on the domain $\Omega=(0,L)^2\subset\mathbb{R}^2$, the SH equation satisfies the energy dissipation law as follows
\begin{align}\label{def: continuous energy dissipation law}
\frac{\mathrm{d}}{\mathrm{d}t}E(u)=\left(\dfrac{\delta E}{\delta u},\partial_tu\right)=-\left\|\frac{\delta E}{\delta u}\right\|^2 \le 0, \quad\text{for}~t>0,
\end{align}
where $(\cdot,\cdot)$ and $\lVert \cdot \rVert$ are the $L^2$ inner product and the associated norm. Such an energy dissipation rule, also known as the energy stability plays a crucial role in designing stable numerical schemes for the phase field models in long time simulation \cite{Du1991SINA,XuTang2006SINA}.

Over the past two decades, the fractional differential equations have attracted much attention due to their superiority in the simulation of various materials and processes with memory and hereditary properties \cite{Qureshi,chenwen2006,Cartea}. There is also a tremendous amount of effort has been put into the applications of the fractional type phase field models, see e.g., \cite{zhaojcns2019,shamseldeen,Liz2017,songfcmame2016}, and a vital issue among these investigations is the energy dissipation property. For instance, Tang et al. \cite{tangtsijsc2019} showed that the energies of the time fractional Allen-Cahn (TFAC) equation and the time fractional Cahn-Hilliard (TFCH) equation with Caputo time fractional derivative are bounded by the initial energies. The boundedness of the discrete energies were also established for several finite difference schemes in which the uniform-step L1 formula was utilized for the approximation of Caputo derivative. Quan et al. \cite{quancsiam} defined a nonlocal energy $E_{\omega}(u):=\int_{0}^{1}\omega(\theta)E(\theta t)\mathrm{d}\theta$, which was proved to be dissipative under a mild restriction of the weight function $\omega(t)$, for the TFAC equation and the TFCH equation. Recently, for the time fractional phase field models, Liao et al. \cite{liaozhuwangnmtma,YangJCP2022} constructed some energy dissipation laws, which are asymptotically compatible with the classical model when the fractional order tends to the first order. More related works can be found in \cite{quanwb,quanarxiv}.

Recently, the existing works for the time fractional Swift-Hohenberg (TFSH) equation are mainly devoted to the approximated analytical solutions \cite{vp2020,PD2019,mm2013,kan2011,Rashid2021}, while the research of the numerical treatment is limited. Zahra et al. \cite{zahra2019} proposed a rational spline-nonstandard finite difference scheme for the TFSH equation, where the Gr\"{u}nwald-Letnikov (GL) formula was applied for the discretization of Riemann–Liouville fractional derivative. The Fourier method showed the unconditionally stability and the first order convergence of the scheme. In addition, for solving the cubic-quintic standard and modified TFSH equations with Caputo fractional derivative, several schemes based on the exponential fitting technique in space and the uniform-step GL formula as well as the L1 formula on the graded mesh considering the initial singularity in time were proposed in \cite{zahra2020}. The first order convergence in time was proved for the GL scheme, and the influence of the orders of time fractional derivative and the length on the solution were illustrated graphically. Besides, a Fourier spectral method was proposed for the Swift-Hohenberg equation with a nonlocal nonlinearity in \cite{nonlocalSH}. However, the energy dissipation laws of the numerical schemes are not considered in the existing literatures. Particularly, designing the suitable numerical scheme on the nonuniform time grid is significant for the simulation of the phase field models, whereas the corresponding theoretical analysis is tough.

In this paper, we are concerned with the TFSH equation \cite{kan2011} as follows 
\begin{align}\label{def:the TFSH equation}
    ^C_0D_t^{\alpha}u=-\mu\quad\text{with}~\mu:=(1+\Delta)^2u+f(u),
\end{align}
subjected to the periodic boundary conditions, the initial value $u(\mathrm{x},0)=u_0(\mathrm{x})$, and the bulk force $f(u)=F'(u)$. The Caputo fractional derivative of order $\alpha\in(0,1)$ is defined by
\begin{align*}
	^C_0D_t^{\alpha}v:=\mathcal{I}^{1-\alpha}_tv'(s)
\end{align*}
where $\mathcal{I}^{\beta}_t$ is the fractional Riemann-Liouville integral,
\begin{align*}
	\mathcal{I}^{\beta}_tv(t)=\int_{0}^{t}\omega_\beta(t-s)v(s)\mathrm{d}s\quad\text{with}~\omega_\beta(t)=t^{\beta-1}/\Gamma(\beta).
\end{align*}

The TFSH equation \eqref{def:the TFSH equation} has the following energy dissipation law,
\begin{align}\label{TF energy law}
	\dfrac{\mathrm{d}\mathcal{E_{\alpha}}}{\mathrm{d}t}+\dfrac{1}{2}\omega_{\alpha}(t)\lVert \mu \rVert^2 \le 0,
\end{align}
where the modified variational energy is defined as in \cite{liaottsiamjsc2021}
\begin{align}\label{TF energy}
	\mathcal{E}_{\alpha}[u]:=E[u]+\dfrac{1}{2}\mathcal{I}^{\alpha}_t\lVert \mu \rVert^2.
\end{align}
It is easily to see that \eqref{TF energy law} recovers the energy law \eqref{def: continuous energy dissipation law} of the classic SH equation as $\alpha\to1^-$, that is
\begin{align*}
	\dfrac{\mathrm{d}E}{\mathrm{d}t}+\lVert \mu \rVert^2 \le 0.
\end{align*}

We concentrate on establishing a variable-step L1 scheme for the TFSH equation \eqref{def:the TFSH equation}, the main contributions are listed in the following
\begin{itemize} 
\item By taking advantage of the discrete gradient structure of the L1 formula on the nonuniform mesh, the numerical scheme is proved to satisfy the discrete energy dissipation law and thus reliable for the long time simulation. 
\item The $L^2$ norm error estimate of the numerical scheme is given in virtue of the global consistency error analytical technique and the fractional Gr\"{o}nwall inequality. 
\item The restriction on the temporal mesh in our analysis is mild which permits the utilization of the adaptive time-stepping strategy in practice. Thus, the effects for the order of the fractional derivative and the cubic term in the energy function on the pattern formation are investigated numerically. Besides, the discrete energies of the TFSH model \eqref{def:the TFSH equation} are also shown under different parameters. 
\end{itemize}

The remainder of the paper is organized as follows. In section 2, the variable-step L1 scheme is constructed for the TFSH equation \eqref{def:the TFSH equation} and proved to be unique solvable. In section 3, the discrete energy dissipation law of the numerical scheme is presented in virtue of some discrete kernel tools, then the $L^\infty$ norm boundedness of the discrete solution is obtained immediately. The $L^2$ norm convergence analysis is derived in section 4. Several numerical examples are included in section 5 to verify the accuracy and the energy dissipation of the proposed scheme, the efficiency of the adaptive time-stepping strategy is further demonstrated.

\section{The variable-step L1 scheme}
This section is devoted to the construction of the variable-step L1 scheme for the TFSH equation \eqref{def:the TFSH equation}, and the unique solvability of the numerical scheme is rigorously demonstrated.

\subsection{The fully discrete scheme}

Take a general time grid $0=t_0<t_1<t_2<\cdots<t_N=T$ with a terminated time $T>0$. Denote the variable time-steps $\tau_k:=t_k-t_{k-1}$ for $1\le k\le N$ and the maximum step size $\tau:=\max_{1\le k\le N}\tau_k$. Let the adjacent time-step ratios $r_k:=\tau_k/\tau_{k-1}$ for $2\le k\le N$. For any grid function $v^n=v(t_n)$, denote $\diff v^n:=v^n-v^{n-1}$ and $\partial_\tau v^n:=\diff v^n/\tau_n$. The well-known L1 formula of Caputo derivative reads as
\begin{align}\label{def:L1 formula}
	D_\tau^{\alpha}v^n:=\sum_{k=1}^{n}a_{n-k}^{(n)}\diff v^n,\quad \text{for}~n \ge 1,
\end{align}
where the discrete kernels $a_{n-k}^{(n)}$ are defined by
\begin{align}\label{the defination of a}
	a_{n-k}^{(n)}:=\dfrac{1}{\tau_k}\int_{t_{k-1}}^{t_k}\omega_{1-\alpha}(t_n-s)\mathrm{d}s,\quad \text{for}~ 1\le k \le n.
\end{align}
The positivity and the monotonicity of kernel $\omega_{1-\alpha}(t)$ give the following properties of $a_{n-k}^{(n)}$,
\begin{align}\label{monotonic}
	a_{0}^{(n)}> a_{1}^{(n)} > \cdots> a_{n-1}^{(n)}>0,\quad \text{for}~1\le n \le N.
\end{align}

As for the spatial discretization, the finite difference method is employed. Set the space-sizes $h_x=h_y=h:=L/M$ with a positive integer $M$, and denote $x_i=ih$, $y_j=jh$ and ${\mathrm{x}_h}=(x_i, y_j)$.
The discrete spatial grid
$\Omega_h:=\big\{{\mathrm{x}_h}=(x_i, y_j)~|~1\le i, j\le M-1\big\}$
and $\bar{\Omega}_h:=\big\{{\mathrm{x}_h}=(x_i, y_j)~|~0\le i, j\le M\big\}$.
The $L$-periodic function space is defined as
$$\mathcal{V}_h:=\big\{v_h=v({\mathrm{x}_h})~|~{\mathrm{x}_h}\in \bar{\Omega}_h
\;\text{and $v_h$ is $L$-periodic in each direction}\big\}.$$
For any grid function $v\in \mathcal{V}_h$, some difference notations are introduced as follows:
$\delta_xv_{i+\frac{1}{2},j}=(v_{i+1,j}-v_{ij})/h$, $\delta_xv_{i-\frac{1}{2},j}=(v_{ij}-v_{i-1,j})/h$
and $\delta^2_xv_{ij}=(\delta_xv_{i+\frac{1}{2},j}-\delta_xv_{i-\frac{1}{2},j})/h.$
The discrete notations $\delta_yv_{i,j+\frac{1}{2}}$ and $\delta^2_yv_{ij}$ can be defined similarly.
Also, the standard discrete Laplacian operator $\Delta_hv_{ij}=\delta^2_xv_{ij}+\delta^2_yv_{ij}$ and
the discrete gradient vector $\nabla_hv_{ij}=(\delta_xv_{ij},~\delta_yv_{ij})^T$.

Putting the equation at $(\mathrm{x}_h,t_n)$, then we develop the following variable-step L1 scheme for the TFSH equation \eqref{def:the TFSH equation},
\begin{align}\label{def:discrete scheme}
	D_\tau^{\alpha}u^n_h=-\mu^n_h~~\text{with} ~~~\mu^n_h:=(1+\Delta_h)^2u^n_h+f(u^n_h),\quad\mathrm{x}_h\in\Omega_h\;,1\le n\le N,
\end{align}
subjected to the periodic boundary conditions and the initial value
$u_h^0=u_0(\mathrm{x}_h)$.

For any grid functions $v,w\in\mathcal{V}_h,$ define the discrete inner product
$\myinner{v,w}:=h^2\sum_{{\mathrm{x}_h}\in\Omega_h}v_h w_h$, the associated $L^2$ norm
$\mynorm{v}_{l^2}:=\sqrt{\myinner{v,v}}$ and the $L^q$ norm
$\mynorm{v}_{l^q}:=\sqrt[q]{h^2\sum_{{\mathrm{x}_h}\in\Omega_h}\abs{v_h}^q}$ for $v\in \mathcal{V}_h$.
Here and hereafter, we write $\|v\|:=\|v\|_{l^2}$ and the discrete $L^{\infty}$ norm $\mynorm{v}_\infty:=\max_{{\mathrm{x}_h}\in\Omega_h}|v_h|$ for simplicity. In addition, the discrete Green's formula with the periodic boundary conditions yields
$\myinner{-\Delta_h v,w}=\myinner{\nabla_hv,\nabla_hw}$ and $\myinner{\Delta_h^2v,w}=\myinner{\Delta_hv,\Delta_hw}$.

\subsection{Unique solvability}
Now the unique solvability of the variable-step L1 scheme \eqref{def:discrete scheme} will be proved via the Brouwer fixed-point theorem \cite{Akrivis1993}. For any fixed index $n\ge1$, define the map $\Pi_n:\mathcal{V}_{h}\rightarrow \mathcal{V}_{h}$ as follows:
\begin{align}\label{eq: nonlinear map}
	\Pi_n(w_h):=a_0^{(n)}w_h-g_h^{n-1}+(1+\Delta_h)^2w_h+f(w_h),\quad\mathrm{x}_h\in\bar{\Omega}_h,
\end{align}
where $g_h^{n-1}=a_0^{(n)}u_h^{n-1}-\sum_{k=1}^{n-1}a_{n-k}^{(n)}\diff u^k$. It is obviously that the equation $\Pi_n(u^n)=0$ is equivalent to the numerical scheme \eqref{def:discrete scheme}. Thus, the solvability of the proposed scheme \eqref{def:discrete scheme} can be verified via the equation \zhur{$\Pi_n(w_h)=0$ in the following theorem.}

\begin{theorem} \label{thm:the Unique solvablity}
	If the maximum time step size $\tau$ satisfies
	\begin{align}\label{def:the condition of tau}
		\tau \le \sqrt[\alpha]{\frac{3}{\Gamma(2-\alpha)(4\mathrm{g}^2+3\epsilon)}},
	\end{align}
	then the numerical scheme \eqref{def:discrete scheme} is uniquely solvable.
\end{theorem}

\begin{proof} The existence of the solution will be shown firstly.
	Suppose that $u^{n-1}, u^{n-2},\cdots, u^{1}$ have been determined, taking the inner product of $\Pi_n(w)$ with $w$ yields
	\begin{align*}
		\myinner{\Pi_n(w), w}=a_0^{(n)}\myinner{w, w}+\myinner{(1+\Delta_h)^2w, w}
		+ \myinner{f(w), w}-\myinner{g^{n-1}, w},\quad\text{for}~n\ge1.
	\end{align*}
	For the second term on the right hand side, the discrete Green's formula implies that
	\begin{align*}
		\myinner{(1+\Delta_h)^2w, w}=\myinner{(1+\Delta_h)w, (1+\Delta_h)w}\ge 0.
	\end{align*}
	Then, in virtue of Cauchy-Schwarz inequality and the time step restriction, one gets
	\begin{align*}
		\myinner{\Pi_n(w), w}\geq&a_0^{(n)}\|w\|^2+\|w\|_4^4-\mathrm{g}\myinner{w^2,w}-\epsilon\|w\|^2
		-\|g^{n-1}\|\cdot\|w\|\\
		\geq&\big(a_0^{(n)}-\epsilon\big)\|w\|^2+\|w\|_4^4-\big(\frac{\mathrm{4g^2}}{3}\|w\|^2+\frac{3}{16}\|w\|_4^4\big)-\|g^{n-1}\|\cdot\|w\|\\
		\ge&\big(a_0^{(n)}-\epsilon-\frac{4\mathrm{g}^2}{3}\big)\|w\|^2-\|g^{n-1}\|\cdot\|w\|=0,
	\end{align*}  
	provided that $\|w\|=\frac{1}{a_0^{(n)}-\epsilon-4\mathrm{g}^2/3}\|g^{n-1}\|$. Hence there exists a $w^*$ such that $\Pi_n(w^*)=0$ according to the Brouwer fixed-point theorem, which implies the numerical scheme \eqref{def:discrete scheme} is solvable.
	
	Then, we are going to prove the uniqueness of the discrete solution. Suppose $w_h$ and $v_h$ are two solutions of the numerical scheme \eqref{def:discrete scheme}. Denote $\zeta_h=w_h-v_h$, then the following equation holds for $\zeta_h$,
	\begin{align}
		a_0^{(n)}\zeta_h+(1+\Delta_h)^2\zeta_h+f(w_h)-f(v_h)=0.\label{eq: Error nonlinear map}
	\end{align}
	Taking the inner product of \eqref{eq: Error nonlinear map} with $\zeta$, one has
	\begin{align}
		a_0^{(n)}\|\zeta\|^2+\|(1+\Delta_h)\zeta\|^2+\myinner{f(w)-f(v), \zeta}=0.\label{3.16}
	\end{align}
	For the nonlinear term in \eqref{3.16}, the further estimation gives
	\begin{align*}
		\myinner{f(w)-f(v), \zeta}&=\myinner{w^3-v^3, \zeta}-\mathrm{g}\myinner{w^2-v^2, \zeta}-\epsilon\|\zeta\|^2\\
		&=\myinner{w^2+wv+v^2-\mathrm{g}w-\mathrm{g}v, \zeta^2}-\epsilon\|\zeta\|^2\\
		&\ge \frac{1}{2}\myinner{(w-\mathrm{g})^2+(v-\mathrm{g})^2, \zeta^2}-\mathrm{g}^2\|\zeta\|^2-\epsilon\|\zeta\|^2.
	\end{align*}
	Substituting the above inequality into \eqref{3.16} yields
	$$(a_0^{(n)}-\mathrm{g}^2-\epsilon)\|\zeta\|^2\leq0.$$
	Thus $\|\zeta\|=0$ due to the time step restriction,
	which implies that the variable-step L1 scheme \eqref{def:discrete scheme} has a unique solution.	
\end{proof}

\section{Discrete energy dissipation law}
In this section, the energy dissipation property of the variable-step L1 scheme \eqref{def:discrete scheme} will be demonstrated. To this end, we use the novel tool proposed in \cite{liaoyanyg,liaozhuwangnmtma} which provides a framework for energy stability and convergence analysis of the time-stepping scheme for the time fractional phase field models.

The important ingredient of this analysis technique is the discrete complementary convolution (DCC) kernels, which is generated by the following recursive procedure
\begin{align*}
	p^{(n)}_0:=\frac{1}{a_{0}^{(n)}}\quad \mathrm{and} \quad p_{n-k}^{(n)}:=\frac{1}{a_{0}^{(k)}}
	\sum_{j=k+1}^n(a_{j-k-1}^{(j)}-a_{j-k}^{(j)})p_{n-j}^{(n)}
	\quad \text{for $1\le k\le n-1$}.
\end{align*}
%\begin{align*}
%	\theta_{0}^{(n)}:=\frac{1}{a_{0}^{(n)}}
%	\quad \mathrm{and} \quad
%	\theta_{n-k}^{(n)}:=-\frac{1}{a_{0}^{(k)}}
%	\sum_{j=k+1}^n\theta_{n-j}^{(n)}a_{j-k}^{(j)}
%	\quad \text{for $1\le k\le n-1$}.
%\end{align*}
%One can easily verify the following discrete orthogonal identity
%\begin{align}\label{the defination of DOC}
%	\sum_{j=k}^n\theta_{n-j}^{(n)}a_{j-k}^{(j)}\equiv \delta_{nk}\quad\text{for $1\le k\le n$},
%\end{align}
%where $\delta_{nk}$ is the Kronecker delta symbol. While the DCC kernels $p^{(n)}_{n-k}$ are the summation of the DOC kernels, that is
%\begin{align}\label{def:p_n-k} 
%	p_{n-k}^{(n)}:=\sum_{j=k}^{n}\theta_{j-k}^{(j)}  \quad\text{for $1\le k\le n$}.
%\end{align}
The DCC kernels are proven to be complementary to the discrete L1 kernels $a^{(n)}_{n-k}$, namely
\begin{align*}
	\sum_{j=k}^np_{n-j}^{(n)}a_{j-k}^{(j)}\equiv 1\quad\text{for $1\le k\le n$}.
\end{align*}

The following lemma reveals the discrete gradient structure of the L1 formula \eqref{def:L1 formula} which is essential for the construction of the discrete energy dissipation law of the variable-step L1 scheme \eqref{def:discrete scheme}.
\begin{lemma}{\rm\cite{YangJCP2022}}\label{lem:inequation of an}
	For any real sequence $\{v_k\}_{k=1}^n$, it holds that
	\begin{align*}
		2v_n\sum_{j=1}^na_{n-j}^nv_j\ge a_0^{(n)}v_n^2+\sum_{k=1}^np_{n-k}^{(n)}\left(\sum_{j=1}^ka_{k-j}^{(k)}v_j \right)^2-\sum_{k=1}^{n-1}p_{n-1-k}^{(n-1)}\left(\sum_{j=1}^ka_{k-j}^{(k)}v_j \right)^2.
	\end{align*}
\end{lemma}

%\begin{proof}
%    For any real sequence $\{w_k\}_{k=1}^n$, it holds that\cite{liaozhuwangnmtma}
%	\begin{align}\label{yinliliao}
%		2w_n\sum_{k=1}^n\theta_{n-k}^{(n)}w_k\ge \sum_{k=1}^{n}p_{n-k}^{(n)}w_k^2-\sum_{k=1}^{n-1}p_{n-1-k}^{(n-1)}w_k^2+\frac{1}{\theta_{0}^{(n)}}\left( \sum_{k=1}^{n}\theta_{n-k}^{(n)}w_k\right)^2. 
%	\end{align}
%	We define a sequence $\{v_j\}_{j=1}^n$ by
%	\begin{align}\label{v_j}
%		 v_j:=\sum_{k=1}^{j}\theta_{j-k}^{(j)}w_k.~~\text{for}~1\le j\le n. 
%	\end{align}
%Multiplying the L1 kernels $a_{n-j}^{(n)}$ on both hand sides of \eqref{v_j}, and summing $j$ from $j=1$ to $n$ lead to
%\begin{align*}
%	\sum_{j=1}^{n}a_{n-j}^{(n)}v_j = \sum_{j=1}^{n}a_{n-j}^{(n)}\sum_{k=1}^{j}\theta_{j-k}^{(j)}w_k=\sum_{k=1}^{n}w_k\sum_{j=k}^{n}a_{n-j}^{(n)}\theta_{j-k}^{(j)}=w_n.
%\end{align*}
%With the help of \eqref{yinliliao}, the desired result is obtained.
%\end{proof}

\begin{lemma}{\rm\cite{sunhong2022}}\label{lem:nonlinear-inequality}
	For any $a,b\in \mathbb{R},$ the following inequality holds
	$$(a^3-\mathrm{g}a^2)(a-b){ \ge} \frac{1}{4}(a^4-b^4)-\frac{\mathrm{g}}{3}(a^3-b^3)-\frac{2\mathrm{g}^2}{3}(a-b)^2.$$
\end{lemma}

Let $E[u^n]$ be the discrete version of energy functional \eqref{def:continuous energy}, i.e.,
$$E[u^n]:=\frac{1}{2}\mynormb{(1+\Delta_h)u^n}^2
+\frac{1}{4}\|u^n\|_4^4-\frac{\mathrm{g}}{3}\myinner{(u^n)^2, u^n}-\frac{\epsilon}{2}\|u^n\|^2.$$
Denote the modified discrete energy
\begin{align*}
	\mathcal{E}[u^0]:=E[u^0],~~\text{and}~~\mathcal{E}[u^n]:=E[u^n]+\frac{1}{2}\sum_{j=1}^{n}p_{n-j}^{(n)}\|\mu^j\|^2,\quad\text{for}~n\ge 1.
\end{align*}

Comparing the modified discrete energy with the continuous counterpart in \eqref{TF energy}, we could regard the DCC kernels $p_{n-j}^{(n)}$ as the discrete kernels of the Riemann-Liouville integral. The discrete energy dissipation law with respect to the above variational discrete energy is shown in the following theorem, which indicates that the variable-step L1 scheme is energy stable.

\begin{theorem} \label{thm:the energy dissapation}
	Under the time step restriction \eqref{def:the condition of tau}, the variable-step L1 scheme \eqref{def:discrete scheme} preserves the variational energy dissipation law at each time level,
$$\partial_\tau\mathcal{E}[u^n]\le 0, \quad \text{for}~ 1\le n\le N.$$
\end{theorem}

\begin{proof}
Taking the inner product of \eqref{def:discrete scheme} with $\diff u^n$, we have
\begin{align}\label{Energy-Law-Inner}
	\myinner{D_\tau^{\alpha}u^n, \diff u^n}+\myinner{(1+\Delta_h)^2u^n, \diff  u^n}+\myinner{f(u^n), \diff  u^n}=0,\quad \text{for}~n\ge 1.
\end{align}
According to Lemma \ref{lem:inequation of an}, the first term on the left hand side is handled as follows
\begin{align*}
	\myinner{D_\tau^{\alpha}u^n, \diff u^n} &= \myinner{\sum_{k=1}^{n}a_{n-k}^{(n)}\diff u^k, \diff u^n}\\
	&\ge \frac{1}{2}\sum_{k=1}^{n}p_{n-k}^{(n)}\|\mu^k\|^2-\frac{1}{2}\sum_{k=1}^{n-1}p_{n-1-k}^{(n-1)}\|\mu^k\|^2+\frac{a_0^{(n)}}{2}\|\diff u^n\|^2.
\end{align*}
For the second term, using the summation by parts and the identity $2a(a-b)=a^2-b^2+(a-b)^2,$ we have
\begin{align*}
		\myinner{(1+\Delta_h)^2u^n,\diff  u^n} &=\myinner{(1+\Delta_h)u^n,\diff  (1+\Delta_h)u^n}\\
		&=\frac{1}{2}\braB{\|(1+\Delta_h)u^n\|^2-\|(1+\Delta_h)u^{n-1}\|^2+\|\diff((1+\Delta_h) u^n)\|^2}.%\label{2.13}
\end{align*}
For the nonlinear term, with the help of Lemma \ref{lem:nonlinear-inequality} and the identity $2a(a-b)=a^2-b^2+(a-b)^2$, it yields that
\begin{align*}
	\myinner{f(u^n), \diff  u^n}&=\myinner{(u^n)^3-\mathrm{g}(u^n)^2, \diff u^n}-\epsilon\myinner{u^n, \diff u^n}\\
	&\geq\frac{1}{4}(\|u^n\|_4^4-\|u^{n-1}\|_4^4)-\frac{\mathrm{g}}{3}\Big(\myinner{(u^n)^2, u^n}-\myinner{(u^{n-1})^2, u^{n-1}}\Big)\\
	&\quad-\frac{\epsilon}{2}(\|u^n\|^2-\|u^{n-1}\|^2)-\Big(\frac{2\mathrm{g}^2}{3}+\frac{\epsilon}{2}\Big)\|\nabla_\tau u^n\|^2.
	%\frac{1}{4}(\|\nabla_hu^n\|^4-\|\nabla_hu^{n-1}\|^4)-\frac{1}{2}(\|\nabla_h u^n\|^2-\|\nabla_h u^{n-1}\|^2)-\frac{1}{2}\|\nabla_h (\diff  u^n)\|^2.%\label{2.14}
\end{align*}
Substituting all these into \eqref{Energy-Law-Inner}, the following inequality is valid for $1\le n\le N$,
\begin{align*}
	\left( \frac{a_0^{(n)}}{2}-\frac{2\mathrm{g}^2}{3}-\frac{\epsilon}{2}\right) \|\diff  u^n\|^2
	+{ \mathcal{E}[u^n]}-{ \mathcal{E}[u^{n-1}]}\le0 .
\end{align*}
Thus, the claimed inequality follows immediately by noticing the time step restriction \eqref{def:the condition of tau}.
\end{proof}

\begin{remark}
	The variable-step backward Euler scheme for the classical SH model is given as
	\begin{align*}
		\partial_\tau u^n_h=-\mu^n_h~~\text{with} ~~~\mu^n_h:=(1+\Delta_h)^2u^n_h+f(u^n_h),\quad\mathrm{x}_h\in\Omega_h\;,1\le n\le N,
	\end{align*}
	which is easily verified to satisfy the following energy dissipation law under the time step restriction $\tau\le3/(4\mathrm{g}^2+3\epsilon)$,
	\begin{align*}
		\partial_\tau E[u^n]+\dfrac{{1}}{2}\lVert \mu^n \rVert^2\le0,\quad\text{for}~1\le n \le N.
	\end{align*}
	Noting that when $\alpha\to1^{-}$, the DCC kernels $p^{(n)}_{n-j}\to \tau_j$ and thus we have 
	\begin{align*}
		\partial_\tau\mathcal{E}[u^n]\le0\quad \to \quad \partial_\tau E[u^n]+\dfrac{1}{2}\lVert \mu^n \rVert^2\le0,\quad \alpha\to1^-,
	\end{align*}
	which indicates that the fractional-type discrete energy dissipation law in Theorem \ref{thm:the energy dissapation} is asymptotically compatible with the integer counterpart.
\end{remark}

\begin{lemma}{\rm\cite{sunhong2022}}\label{embedding-equation}
	For any grid function $v\in\mathcal{V}_h,$ it holds that
	\begin{align}
	\|v\|_\infty&\le \tilde{C}_\Omega\brab{\mynorm{v}+\mynorm{(1+\Delta_h)v}},
	\end{align}
	where $\tilde{C}_\Omega$ is a positive constant depending on the size of space domain $\Omega$ but independent of the grid size.
\end{lemma}

It follows from Theorem \ref{thm:the energy dissapation} that the discrete
energy is bounded by the initial energy, that is
\begin{align*}
	E[u^n]\le\mathcal{E}[u^n]\le \mathcal{E}[u^{n-1}]\le \cdots \le \mathcal{E}[u^{0}]= E[u^0],\quad\text{for}~n\ge 1.
\end{align*}	
Furthermore, the following lemma shows that the discrete solution of numerical scheme \eqref{def:discrete scheme} is bounded in $L^\infty$ norm.

\begin{lemma}\label{lem:Bound-Solution}
	If the time step restriction \eqref{def:the condition of tau} holds, then the numerical solution of the variable-step L1 scheme \eqref{def:discrete scheme} is bounded in $L^\infty$ norm, i.e., 
	$\mynormb{u^n}_{\infty}\le c_0$ for $n\ge 1$,
	where $c_0$ is a mesh-independent constant.
\end{lemma}
\begin{proof}
	Utilizing Theorem \ref{thm:the energy dissapation} and Young inequality, it follows from $E[u^n]\le E[u^0]$ that
	\begin{align}
		4{ E[u^0]}\geq& 2\|(1+\Delta_h)u^n\|^2+\|u^n\|_4^4-\frac{4\mathrm{g}}{3}\myinner{(u^n)^2, u^n}-2\epsilon\|u^n\|^2\nonumber\\
		\ge&\, 2\|(1+\Delta_h)u^n\|^2+\|u^n\|_4^4-\frac{2}{9}\mynormb{u^n}_4^4-2\mathrm{g}^2\|u^n\|^2-2\epsilon\|u^n\|^2\nonumber\\
		\ge&\, 2\|(1+\Delta_h)u^n\|^2+\frac{7}{9}\|u^n\|_4^4-2(\mathrm{g}^2+\epsilon)\|u^n\|^2,~~ n\ge1.\label{infinite inequality}
	\end{align}
	For any $a \in \mathbb{R}$, the inequality $\big[a^2-\frac{9}{7}(1+\epsilon+\mathrm{g}^2)\big]^2\ge0$ produces that
	$$\|u^n\|_4^4\ge \frac{18}{7}(1+\epsilon+\mathrm{g}^2)\|u^n\|^2-\frac{81}{49}(1+\epsilon+\mathrm{g}^2)^2|\Omega_h|.$$
	Substituting the above inequality into (\ref{infinite inequality}), we have
	\begin{align*}
		4E[u^0]&\ge 2\|(1+\Delta_h)u^n\|^2+2\|u^n\|^2-\frac{9}{7}(1+\epsilon+\mathrm{g}^2)^2|\Omega_h|\\
		&\ge\big(\|(1+\Delta_h)u^n\|+\|u^n\|\big)^2-\frac{9}{7}(1+\epsilon+\mathrm{g}^2)^2|\Omega_h|.
	\end{align*}
	With the help of Lemma \ref{embedding-equation}, it follows that
	\begin{align*}
		\|u^n\|_\infty&\le \tilde{C}_\Omega(\|(1+\Delta_h)u^n\|+\|u^n\|\big)\\
		&\le \tilde{C}_\Omega\sqrt{4E[u^0]+\frac{9}{7}(1+\epsilon+\mathrm{g}^2)^2|\Omega_h|}\\
		&:=c_0.
	\end{align*}
	This completes the proof.
\end{proof}

\section{$L^2$ norm error estimate}
\setcounter{equation}{0}
In this section, we prove the convergence of the variable-step L1 scheme \eqref{def:discrete scheme}. The time fractional differential equations admit a weak singularity near the initial time \cite{martin2018}. In this paper, it is reasonable to make the following regularity hypothesis for the solution of the TFSH equation \eqref{def:the TFSH equation}: there exists a mesh-independent constant $c_1>0$ such that 
\begin{align}
	\|u\|_\infty\le c_1,\quad\|\partial_{xx}u\|\le c_1,\quad\|\partial_t^{(l)} u\| \le c_1(1+t^{\sigma-l}), \label{regularity}
\end{align}
for $0< t \le T$, $x\in\Omega$ and $l=1,2$, where the regularity parameter $\sigma\in(0,1)\cup(1,2)$.

We adopt a family of nonuniform time meshes which concentrate the time levels near $t=0$ and assume that for a mesh parameter $\gamma\ge 1$ \cite{MclenMustapha}, there exists mesh-independent constant $C_{\gamma}>0$ such that
\begin{align}
	\tau_k\le \tau~\min\{1,C_{\gamma} t_k^{1-1/\gamma}\}~~\text{for}~1\le k\le N ~~\text{and}~~ t_k\le C_{\gamma}t_{k-1}~~\text{for}~2\le k\le N. \label{mesh}
\end{align}

Denote the local consistency error of the variable-step L1 formula $\xi^j={}^C_0D_t^{\alpha}v(t_j)-D_\tau^{\alpha}v^j$. The following lemma estimates the global convolution error with the DCC kernels.
\begin{lemma}{\rm\cite{liaoyanyg}}\label{lem: errors in time}
	For $v\in C^2(0,T]$ and $\int_{0}^{T}t|v_{tt}|\mathrm{d}t<\infty$, the global consistency error of the nonuniform L1 formula \eqref{def:L1 formula} is bounded by
\begin{align*}
	\sum_{j=1}^np_{n-j}^{(n)}|\xi^j|\le c_1\left( \frac{\tau_1^{\sigma}}{\sigma}+\frac{1}{1-\alpha}\max_{2\le k\le n}t_k^{\alpha}t_{k-1}^{\sigma-2}\tau_k^{2-\alpha}\right),\quad \text{for}~1\le n\le N.
\end{align*}	
Moreover, if the mesh satisfies \eqref{mesh}, then 
\begin{align*}
	\sum_{j=1}^np_{n-j}^{(n)}|\xi^j|\le \frac{c_1}{\sigma(1-\alpha)}\tau^{\min\{\gamma\sigma,2-\alpha\}},\quad \text{for}~1\le n\le N.
\end{align*}	
\end{lemma}

We are now in the position to prove the $L^2$ norm convergence of the variable-step L1 scheme \eqref{def:discrete scheme}. Denote the error $e_h^n=U_h^n-u_h^n, { \mathrm{x}}_h\in\bar{\Omega}_h, 0\le n\le N$.
We get the error equations as follows
\begin{align}\label{def:error scheme}
D_\tau^{\alpha}e_h^n+(1+\Delta_h)^2e_h^n+\big(f(U_h^n)-f(u_h^n)\big)=\xi_h^n+\eta_h^n,~~{ \mathrm{x}}_h\in\bar{\Omega}_h,~1\le n\le N, 
\end{align}
with $e_h^0=0, ~{ \mathrm{x}}_h\in \bar{\Omega}_h,$ where $\xi_h^n$ and $\eta_h^n$ denote the local consistency error in time and space, respectively. Let
\begin{align*}
	\tau^*:=\min\left\lbrace \sqrt[\alpha]{\dfrac{3}{\Gamma(2-\alpha)(4\mathrm{g}^2+3\epsilon)}},\dfrac{1}{\sqrt[\alpha]{2\Gamma(2-\alpha)\rho}}\right\rbrace ,
\end{align*}
where $\rho =c_1^2+c_1c_0+c_0^2+\mathrm{g}(c_0+c_1)+\epsilon$.

 \begin{theorem}\label{th3.2}
	Suppose the solution of the TFSH equation \eqref{def:the TFSH equation} satisfies the regularity assumption \eqref{regularity}. If the maximum time step $\tau$ satisfies $\tau\le \tau^*$, then the numerical solution $u^n$ is convergent in $L^2$ norm,
	\begin{align*}
	\|e^n\|\le C_eE_{\alpha}(2\rho t_n^{\alpha}/r^*)\left(  \frac{\tau_1^{\sigma}}{\sigma}+\frac{1}{1-\alpha}\max_{2\le k\le n}t_k^{\alpha}t_{k-1}^{\alpha-2}\tau_k^{2-\alpha}+h^2\right)  \quad\text{for}~1\le n\le N,
	\end{align*}
	where $r^*:=\min_{1\le k\le N}\{1,r_k\}$ is the minimum step ratio, $C_e$ is a positive constant independent of the time steps $\tau_n$ and the space size $h$. $E_{\alpha}(z):=\sum_{k=0}^{\infty}\frac{z^k}{\Gamma(1+k\alpha)}$ is the Mittag-Leffler function.
	
	In particular, when the time mesh satisfies \eqref{mesh}, it holds that
	\begin{align*}
		\|e^n\|\le \frac{C_e}{\alpha(1-\alpha)}E_{\alpha}(2\rho t_n^{\alpha}/r^*)\left(  \tau^{\min\{\gamma\sigma,2-\alpha\}}+h^2\right) \quad\text{for}~1\le n\le N.
	\end{align*}
	Hence the optimal temporal accuracy $O(\tau^{2-\alpha})$ is achieved when $\gamma \ge \max\{1,(2-\alpha)/\alpha\}$.
\end{theorem}

\begin{proof}
Taking the inner product of \eqref{def:error scheme} with $e^n$, one has
\begin{align}\label{e1}
	\myinnerb{D_\tau^{\alpha}e^n,e^n}+\myinnerb{(1+\Delta_h)^2e^n,e^n}+\myinnerb{f(U^n)-f(u^n),e^n}=\myinnerb{\xi^n+\eta^n,e^n}.
\end{align}
With the help of Cauchy-Schwarz inequality, the right-hand side term associated with the truncation error is handled in a straightforward way,
\begin{align*}
	\myinnerb{\xi^n+\eta^n,e^n} \leq (\| \xi^n \|+\|\eta^n \|)\|e^n\|.
\end{align*}
For the first term on the left hand side, using the monotone property \eqref{monotonic} and Cauchy-Schwarz inequality yields
\begin{align*}
	\myinnerb{D_\tau
		^{\alpha}e^n,e^n}&=\myinnerb{\sum_{k=1}^na_{n-k}^{(n)}\diff e^k,e^n}\\
	&=\myinnerb{a_0^{(n)}e^n+\sum_{k=1}^{n-1}(a_{n-k}^{(n)}-a_{n-k-1}^{(n)})e^k-a_{n-1}^{(n)}e^0,e^n}\\
	&\ge a_0^{(n)}\|e^n\|^2+\sum_{k=1}^{n-1}(a_{n-k}^{(n)}-a_{n-k-1}^{(n)})\|e^k\|\|e^n\|-a_{n-1}^{(n)}\|e^0\|\|e^n\|\\
	&=\|e^n\|\sum_{k=1}^na_{n-k}^{(n)}\diff\|e^k\|.
\end{align*}
For the third term on the left hand side, denote
\begin{align*}
	f_h^n=(U_h^n)^2+U_h^nu_h^n+(u_h^n)^2-\mathrm{g}(U_h^n+u_h^n)-\epsilon,
\end{align*} 
it follows from Lemma \ref{lem:Bound-Solution} and the regularity assumption that $$\|{ f^n}\|_\infty\le \rho. $$ Thus, a direct application of Cauchy-Schwarz inequality gives
\begin{align*}
	\myinnerb{f(U^n)-f(u^n),e^n}=\myinner{f^ne^n,e^n}\le
	\rho \| e^n \|^2.
\end{align*}
Noticing that $\myinnerb{(1+\Delta_h)^2e^n,e^n}=\|(1+\Delta_h)e_h^n\|^2 \ge 0$, then we have 
\begin{align*}
	\sum_{k=1}^na_{n-k}^{(n)}\diff\|e^k\|\le \rho\|e^n\|+\|\xi^n\|+\|\eta^n\|.
\end{align*}
With the help of the discrete fractional Gr\"{o}nwall inequality \cite{liaoyanyg}, when the maximum time step size $\tau \le\tau^*$, it holds that
\begin{align*}
\|e^n\| &\le 2E_{\alpha}(2\rho t_n^{\alpha}/r^*)\left(\lVert e^0 \rVert +\max_{1\le k\le n}\sum_{j=1}^kp_{k-j}^{(k)}\|\xi^n\|+\omega_{1+\alpha}(t_n)\max_{1\le k\le n}\|\eta^n\| \right)
\quad\text{for}~1\le n\le N.
\end{align*}
Using of the regularity assumption \eqref{regularity} and Lemma \ref{lem: errors in time} produces the desired estimate, and the proof is completed.
\end{proof}

\section{Numerical experiments}
Some numerical simulations are presented in this section to support the theoretical results. To reduce the computational cost and storage, the sum-of-exponentials technique \cite{jiangshidong} for the Caputo derivative with the absolute tolerance error $10^{-12}$ is employed. The nonlinear system of equations is solved by the fixed-point iteration method with the termination error $10^{-12}$. 

The temporal accuracy order of the variable-step L1 scheme \eqref{def:discrete scheme} is verified. Denote the discrete $L^2$ norm error $e(N):=\|U^N-u^N\|,$ the order of convergence in time direction is
defined by 
$$\mathrm{Order} \approx \frac{{\rm log}(e(N)/e(2N))}{{\rm log}(\tau(N)/\tau(2N))}.$$

\begin{example}\label{example1}
	{(Temporal accuracy test)}
	Consider the TFSH model with a forcing term $g(\mathbf{x},t)$, i.e.,  $^C_0D_t^{\alpha}u+(1+\Delta)^2u+f(u)=g(\mathbf{x},t)$. We solve it on a rectangular domain $\Omega=[0,2\pi]^2$ with $T=1$ and the parameters $\mathrm{g}$=0.1, $\epsilon$=0.5. The exact solution is set as  $$u(x,y,t)=\frac{t^{\sigma}}{\Gamma(1+\sigma)}\sin x\sin y,~~(x,y) \in \Omega,~~0\le t\le T,$$
	where $\sigma \in (0,1)$ is a regularity parameter. 
\end{example}

The spatial domain is divided into a $256^2$ uniform mesh. The time interval $[0,T]$ is divided into two parts: $[0,T_0]$ and $[T_0,T]$, with total $N$ subintervals, where $T_0=\min\{1/\gamma,T\}$. The graded mesh $t_k=T_0(k/N_0)^{\gamma}$ is applied to the first part, where $N_0=\left[\frac{N}{T+1-\gamma^{-1}}\right]$, and the random time mesh with $\tau_{N_0+k}:=(T-T_0)\epsilon_k/S_1$ for $1\le k\le N_1$ are used in the remainder interval, where $N_1=N-N_0$, $S_1=\sum_{k=1}^{N_1}\epsilon_k$ and $\epsilon_k \in (0,1)$ are random numbers.

The numbers of temporal intervals are varied by $N=20,40,80,160$ for $\alpha=0.5$ and $\alpha=0.8$. The optimal grading parameter $\gamma_{\text{opt}}=(2-\alpha)/\sigma$ is suggested to achieve the optimal order $O(\tau^{2-\alpha})$. It is seen from Table \ref{alpha=0.5} and Table \ref{alpha=0.8} that when $\gamma<\gamma_{\text{opt}}$, the convergence order is $O(\tau^{\gamma \sigma})$, whereas when $\gamma\ge\gamma_{\text{opt}}$, the convergence order is $O(\tau^{2-\alpha})$, which shows the sharpness of our theoretical result.
\begin{table}[tbh!]
	\begin{center}
		\caption{Errors and convergence orders in time with $\alpha$=0.5, $\sigma$=0.3 } 
		\vspace*{0.3pt}
		\def\temptablewidth{1\textwidth}
		{\rule{\temptablewidth}{1pt}}
		\begin{tabular*}{\temptablewidth}{@{\extracolsep{\fill}}ccccccc}
			\hline
			\multicolumn{1}{c}{}  & \multicolumn{2}{c}{$\gamma$=4} & \multicolumn{2}{c}{$\gamma$=5} & \multicolumn{2}{c}{$\gamma$=6} \\ \cline{2-3}\cline{4-5}\cline{6-7} 
			\multicolumn{1}{c}{$N$} & $e(N)$             & Order         & $e(N)$           & Order           & $e(N)$          & Order          \\ \hline
			20                    & 9.32E-03         & *             & 4.13E-03             & *               & 3.91E-03             & *              \\
			40                    & 4.27E-03          & 1.13        & 1.56E-03             & 1.41               & 1.43E-03             & 1.45              \\
			80                    & 1.86E-03          & 1.20         & 5.69E-04              & 1.45               & 5.16E-04             & 1.47              \\
			160                    & 8.09E-04          & 1.20          & 2.05E-04             & 1.48               & 1.84E-04   & 1.49                         \\ \hline
		\end{tabular*}
		{\rule{\temptablewidth}{1pt}}
		\label{alpha=0.5}
	\end{center}
\end{table}

\begin{table}[tbh!]
	\begin{center}
		\caption{Errors and convergence orders in time with $\alpha$=0.8, $\sigma$=0.3 } 
		\vspace*{0.3pt}
		\def\temptablewidth{1\textwidth}
		{\rule{\temptablewidth}{1pt}}
		\begin{tabular*}{\temptablewidth}{@{\extracolsep{\fill}}ccccccc}
			\hline
			\multicolumn{1}{c}{}  & \multicolumn{2}{c}{$\gamma$=3} & \multicolumn{2}{c}{$\gamma$=4} & \multicolumn{2}{c}{$\gamma$=5} \\ \cline{2-3}\cline{4-5}\cline{6-7} 
			\multicolumn{1}{c}{$N$} & $e(N)$            & Order         & $e(N)$           & Order           & $e(N)$          & Order          \\ \hline
			20                    & 2.34E-02         & *             & 1.56E-02              & *               & 1.48E-02           & *              \\
			40                    &1.26E-02          & 0.89
          & 7.63E-03             & 1.03               & 6.81E-03
            & 1.12              \\
			80                    & 6.77E-03         & 0.90         & 3.45E-03              & 1.14               & 3.10E-03
            & 1.14              \\
			160                    & 3.63E-03          & 0.90          & 1.54E-03              & 1.17               & 1.37E-03    & 1.17    \\               \hline
		\end{tabular*}
		{\rule{\temptablewidth}{1pt}}
		\label{alpha=0.8}
	\end{center}
\end{table}

The second numerical simulation is to show the coarsening dynamics of the TFSH model. Similar to Example \ref{example1}, the time interval $[0,T]$ is divided into two parts: $[0,T_0]$ and $[T_0,T]$. The first part $[0,T_0]$ is treated as in Example \ref{example1} with the grading parameter $\gamma=3$, while for the remainder interval $[T_0,T]$, an adaptive time-stepping strategy related to the change rate of the numerical solution is adopted, cf.\cite{liaozhuwangnmtma},
$$\tau_{ada}=\max\left\lbrace \tau_{\min},\frac{\tau_{\max}}{\sqrt{1+\eta\|\diff u^n/\tau_n\|^2}}\right\rbrace ,$$
where $\eta$ is a regular parameter, $\tau_{\max}$ and $\tau_{\min}$ are the maximum and the minimum size of time steps.
\begin{example}
	We consider the TFSH model on $\Omega=(0,32)^2$ with $\epsilon=0.85$, subjected to the periodic boundary conditions and the initial data
	\begin{align*}
		u(x,y,0)=&0.07-0.02\cos\braB{\frac{2\pi(x-12)}{32}}\sin\braB{\frac{2\pi(y-1)}{32}}\\
		&+0.02\cos^2\braB{\frac{\pi(x+10)}{32}}\sin^2\braB{\frac{\pi(y+3)}{32}}\\
		&-0.01\sin^2\braB{\frac{4\pi x}{32}}\sin^2\braB{\frac{4\pi(y-6)}{32}}.
	\end{align*}
\end{example}

We investigate the effect of the adaptive time-stepping strategy parameter $\eta$. Take $\tau_{\max}=10^{-1}$, $\tau_{\min}=10^{-3}$ and the spatial step size $h=1/3$. The uniform temporal mesh with $\tau=10^{-2}$ and the adaptive time-stepping strategy with $\eta=10,10^2,10^3$ are utilized respectively in the calculations. The evolutions of the original energies, the modified energies and the time steps are demonstrated in Figure \ref{fig: different eta}. Besides, the corresponding number of time levels and the CPU time are listed in Table \ref{CPUalpha=0.8}, which shows that the parameter $\eta$ definitely affects the adaptive time step size, particularly, $\eta\,  (\eta=10)$ generated the minimum fluctuation of the time step. The above findings show that the variable-step L1 scheme captures the changes of the original energy as well as the modified energy accurately and efficiently.

\begin{figure}[tbh!]
	\centering
	\subfigure[energy]{
		\includegraphics[width=2in]{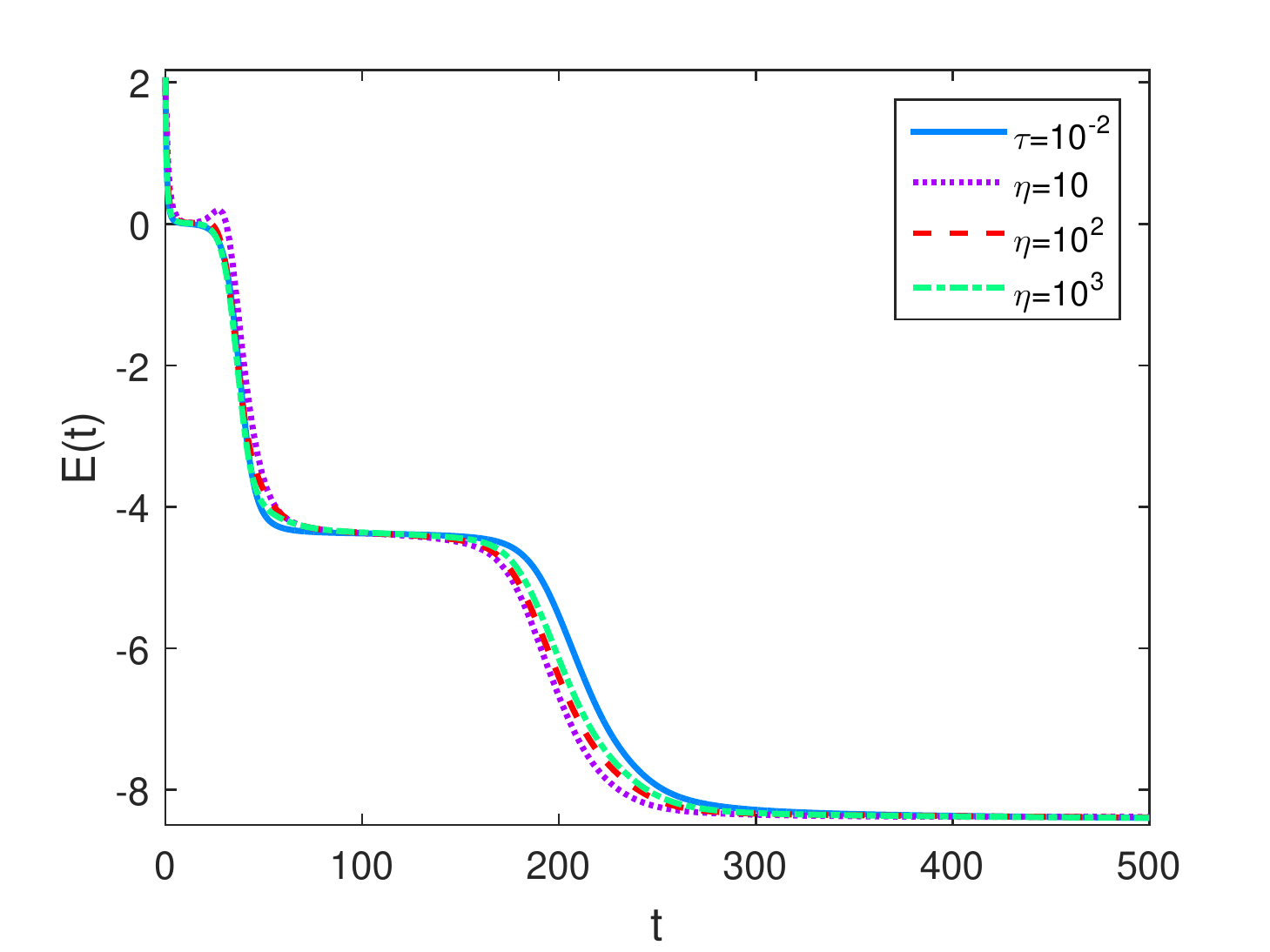}}
	\subfigure[vatiational energy]{
		\includegraphics[width=2in]{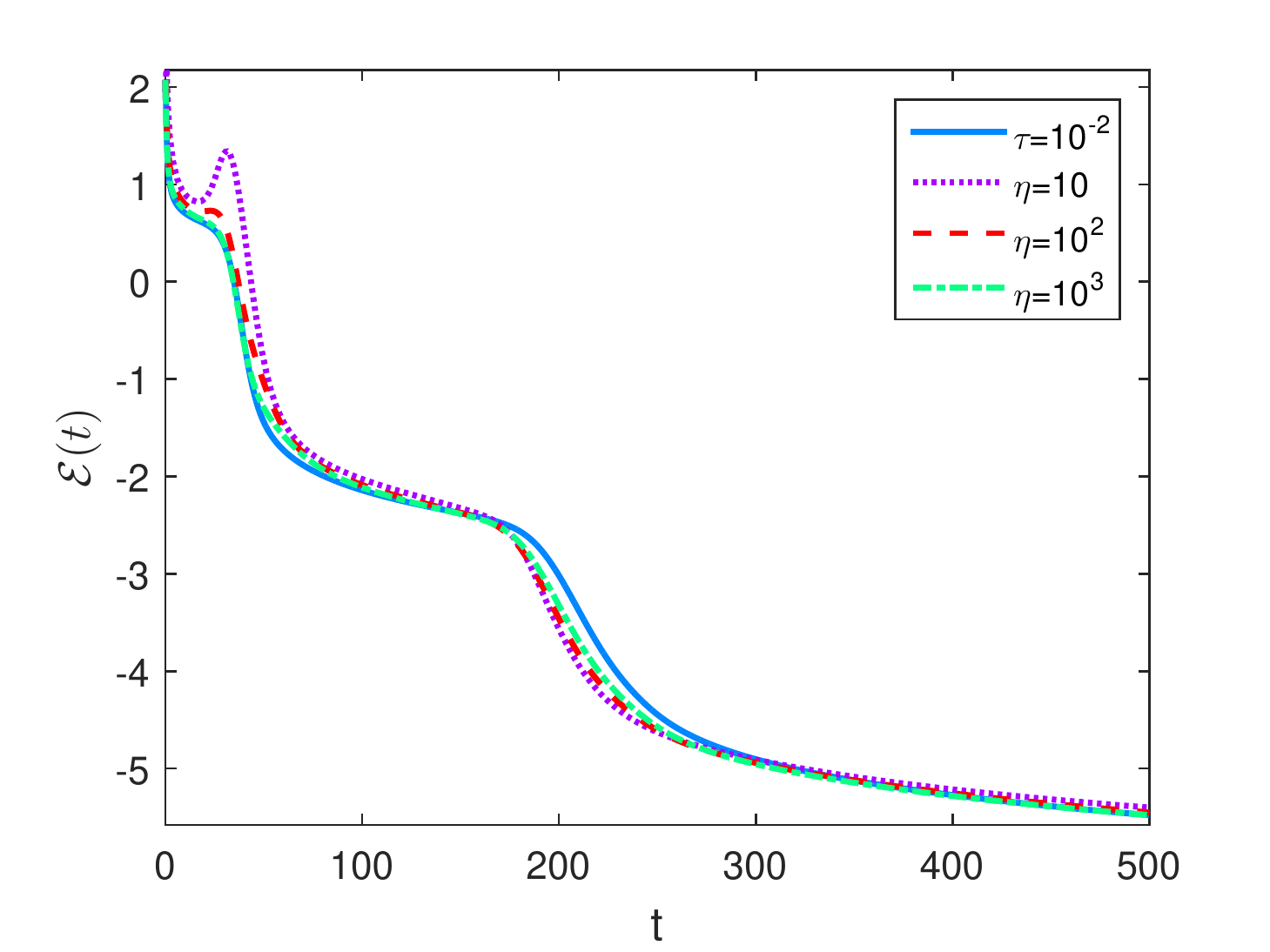}}
	\subfigure[time steps]{
		\includegraphics[width=2in]{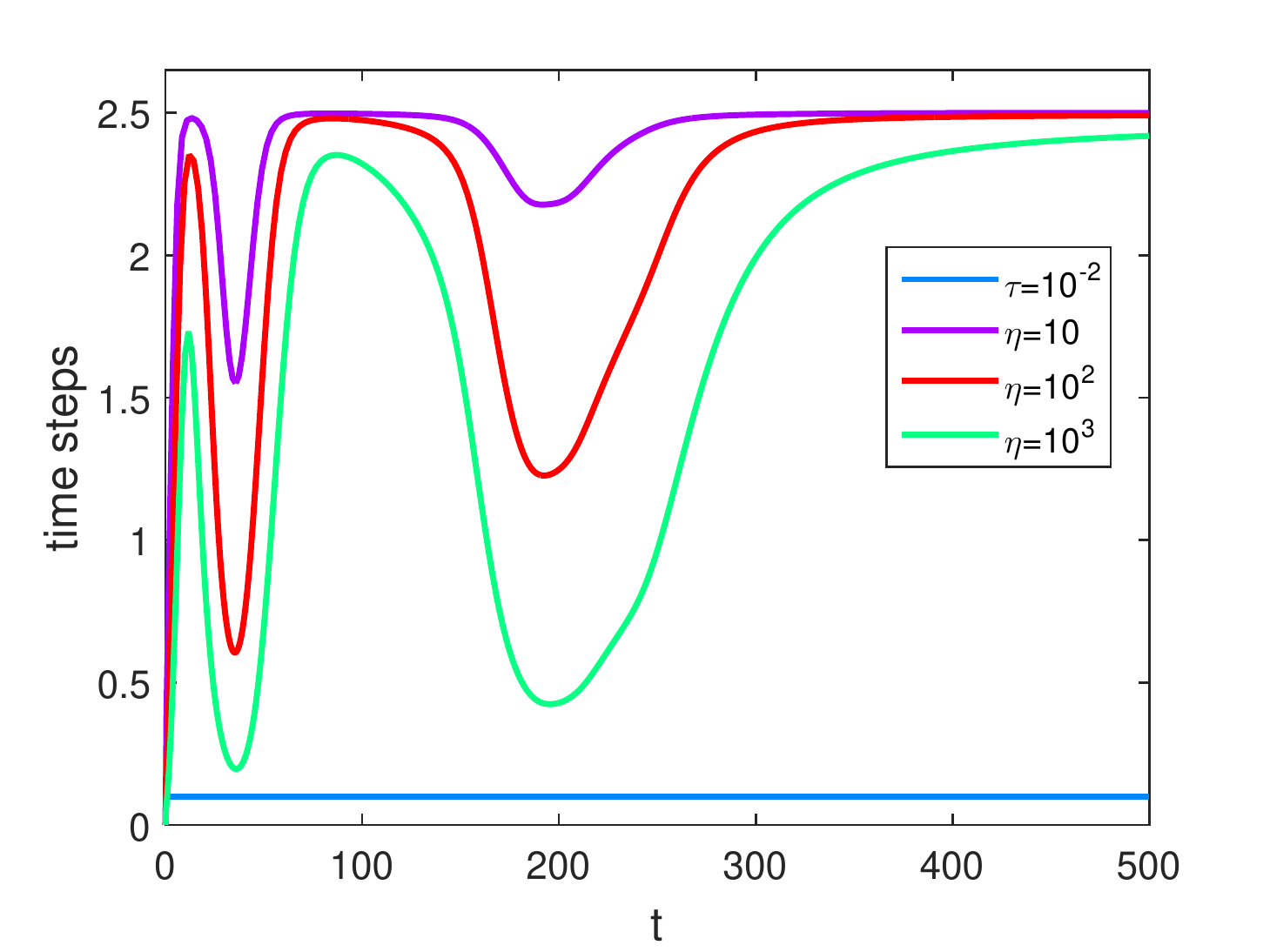}}
%	\subfigure[time steps ratios]{
%		\includegraphics[width=2.4in]{fig/fig_eta_solution/rk-eps-converted-to.pdf}}
	\caption{Evolution of energy by uniform step and adaptive time-stepping strategy with different parameters $\eta$}
	\label{fig: different eta}
\end{figure}

\begin{table}[tbh!]
	\begin{center}
		\caption{Errors and convergence orders in time with $\alpha$=0.8, $\sigma$=0.3 } 
		\vspace*{0.3pt}
		\def\temptablewidth{1\textwidth}
		{\rule{\temptablewidth}{1pt}}
		\begin{tabular*}{\temptablewidth}{@{\extracolsep{\fill}}ccc}
			Time-stepping strategy    & Time levels & CPU(s)  \\
			\cline{1-3}
  uniform step ($\tau=10^{-2}$)   & 5030        & 716 \\ 
  adaptive strategy ($\eta=10$)   & 243         & 34  \\
  adaptive strategy ($\eta=10^2$) & 291         & 41  \\
  adaptive strategy ($\eta=10^3$) & 501         & 70  
		\end{tabular*}
		{\rule{\temptablewidth}{1pt}}
		\label{CPUalpha=0.8}
	\end{center}
\end{table}
%解在不同g条件下的时间演化图
To see the pattern formation of the TFSH model, we perform the proposed numerical scheme with $\tau_{\max}=10^{-1}$, $\tau_{\min}=10^{-3}$ and $\eta=10$ in the adaptive time-stepping strategy for fixed $\alpha=0.6$. Figure \ref{fig:solution_g=1} and Figure \ref{fig:solution_g=0} show the profiles at four observation times with $\mathrm{g}=1$ and $\mathrm{g}=0$, respectively. As $\mathrm{g}=0$, i.e., the effect of cubic term in the energy function is vanished, we can observe the striped pattern, however the hexagonal pattern appeared as $\mathrm{g}=1$. Such phenomenon is also reported for the classic SH model in \cite{leecmame,yangjxkim2022}. The energy stabilities with different $\mathrm{g}$ are displayed in Figure \ref{fig:energy_g}. These numerical evidences suggest that the parameter $\mathrm{g}$ plays a dominated role in the formation of different patterns.

\begin{figure}[tbh!]
	\centering
	\subfigure[t=64]{
		\includegraphics[width=1.4in]{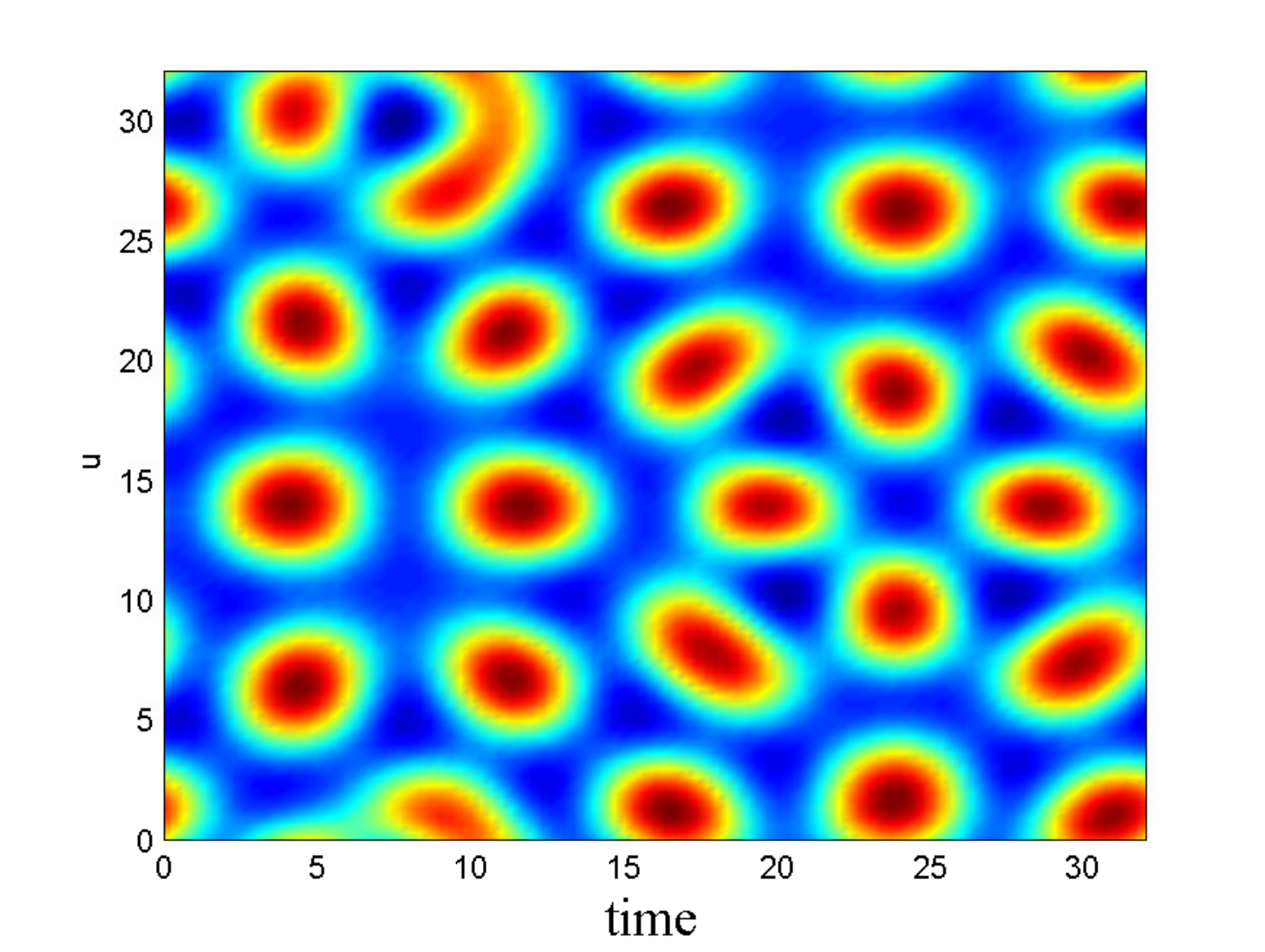}}
	\subfigure[t=128]{
		\includegraphics[width=1.4in]{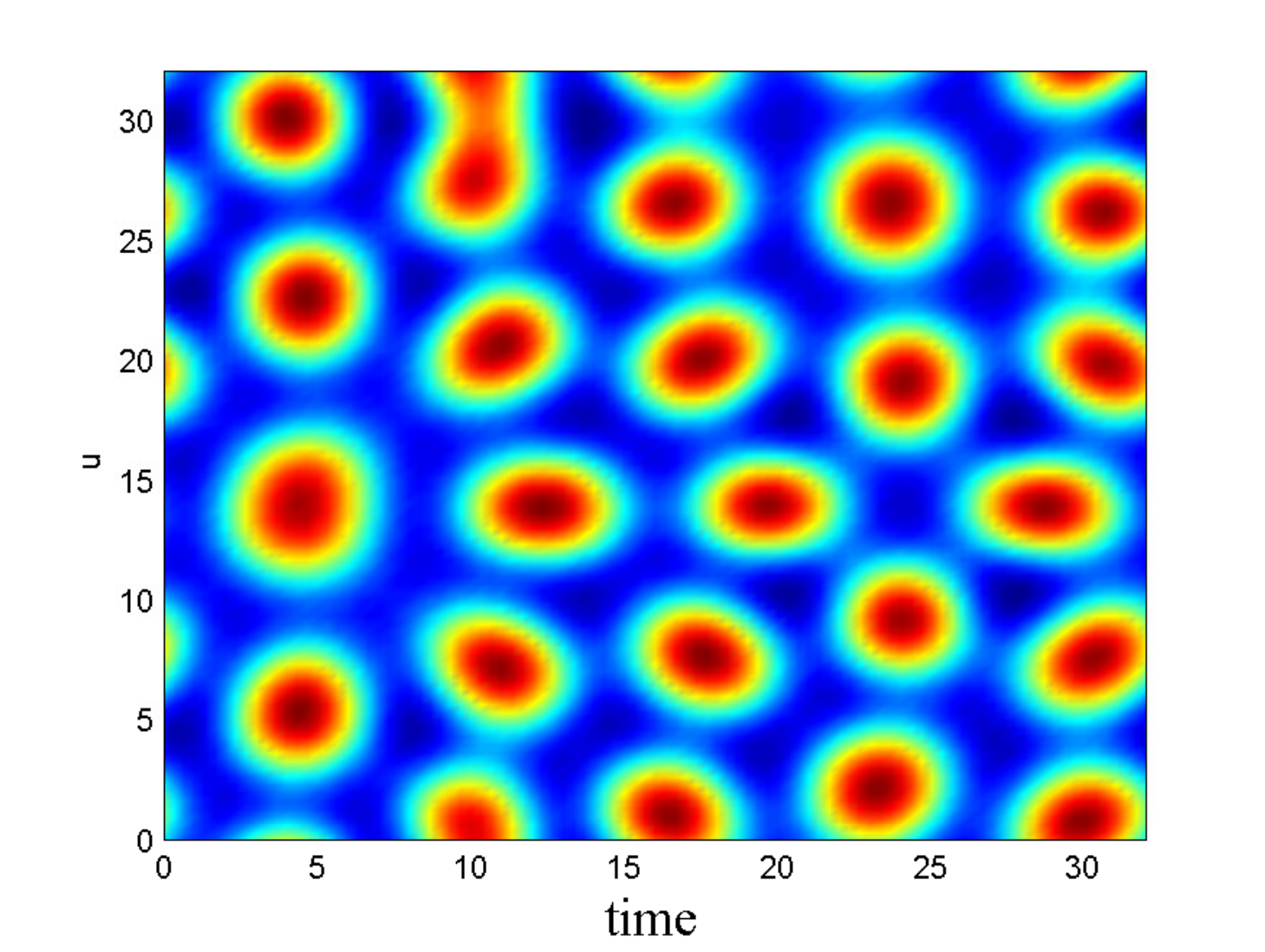}}
	\subfigure[t=256]{
		\includegraphics[width=1.4in]{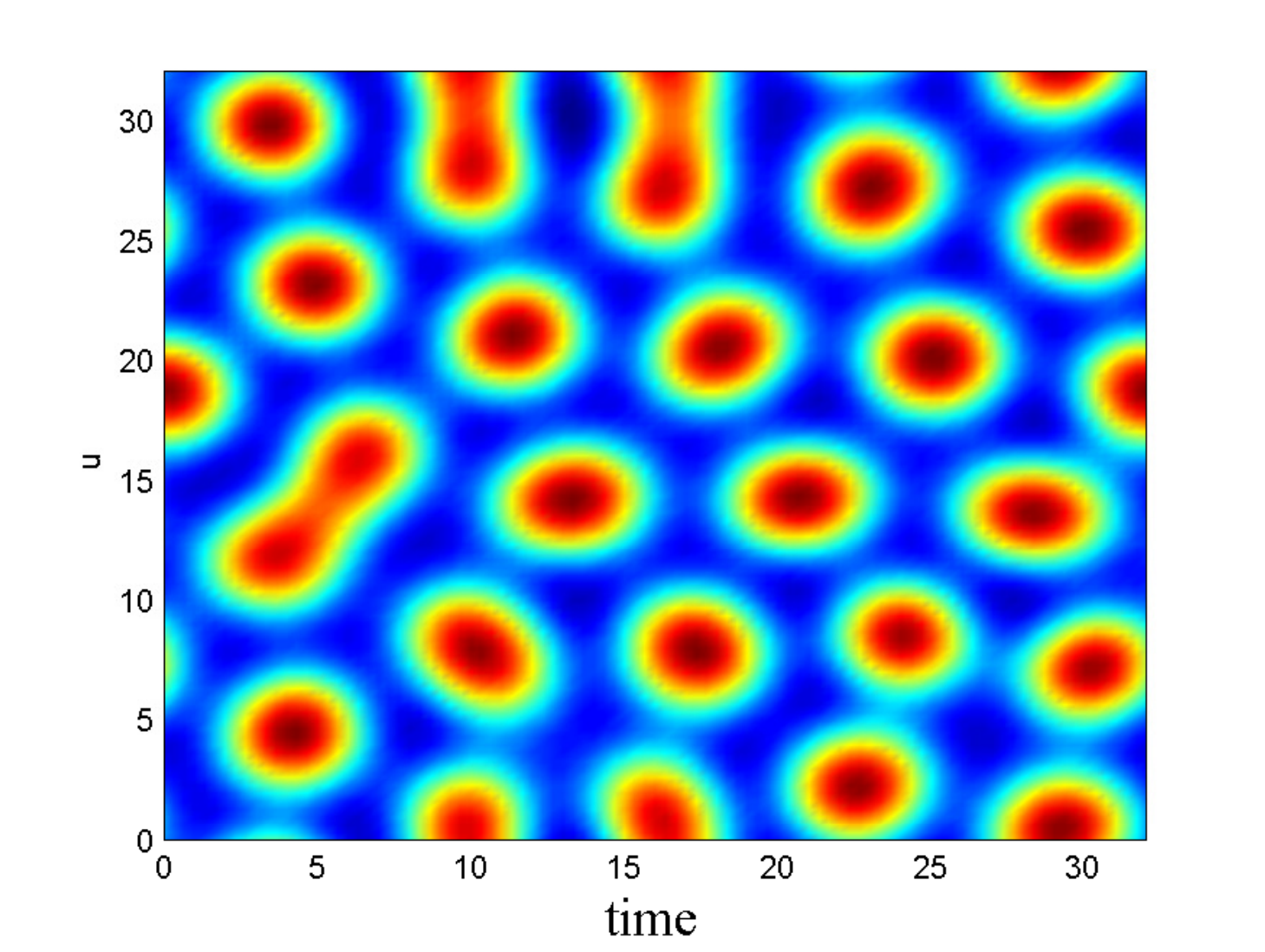}}
	\subfigure[t=512]{
		\includegraphics[width=1.4in]{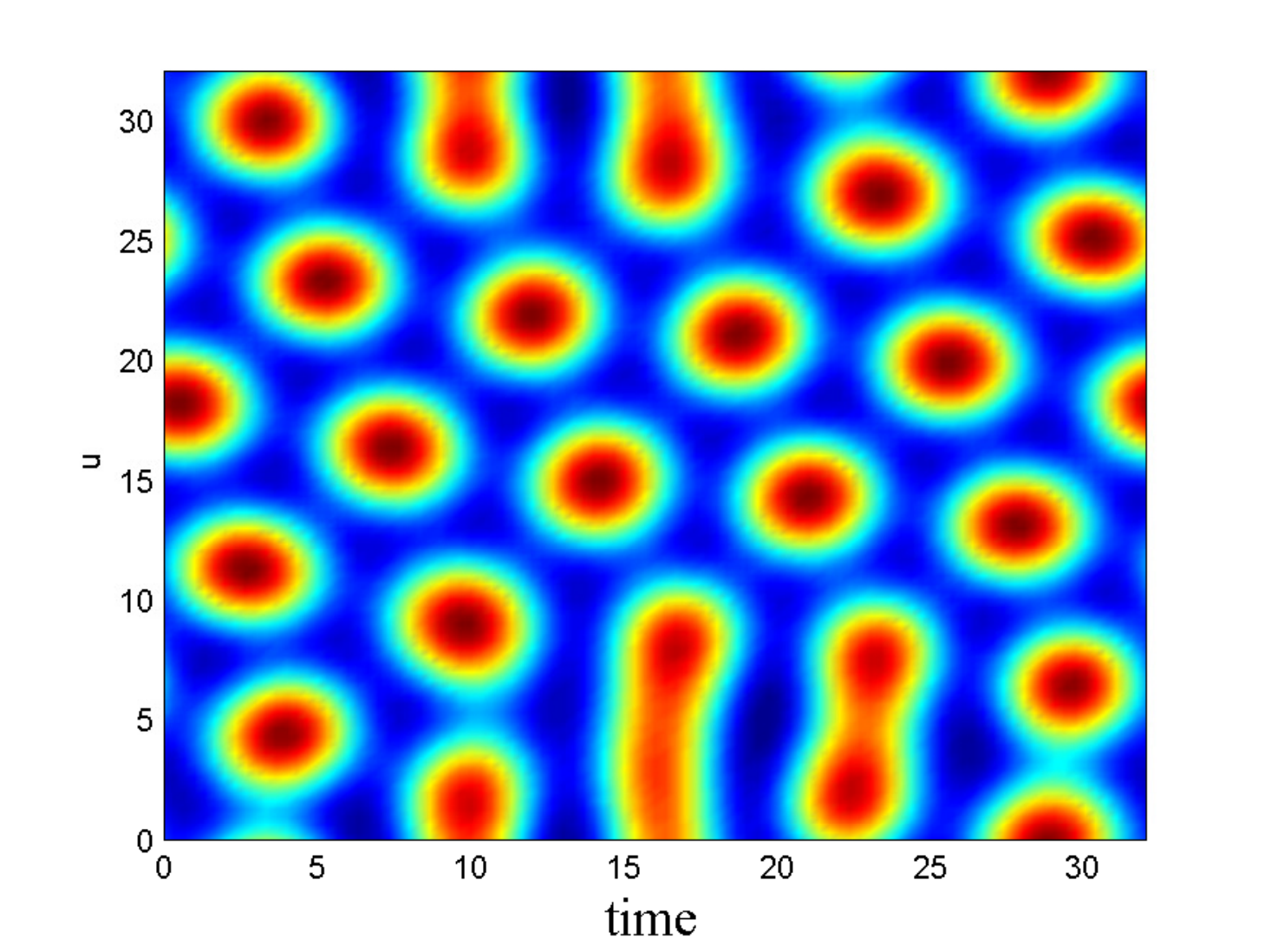}}
	\caption{Time snapshots of TFSH model at $t=64, 128, 256, 512$ with  $\alpha=0.6$ and $\mathrm{g}=1$}
	\label{fig:solution_g=1}
\end{figure}

\begin{figure}[tbh!]
	\centering
	\subfigure[t=64]{
		\includegraphics[width=1.4in]{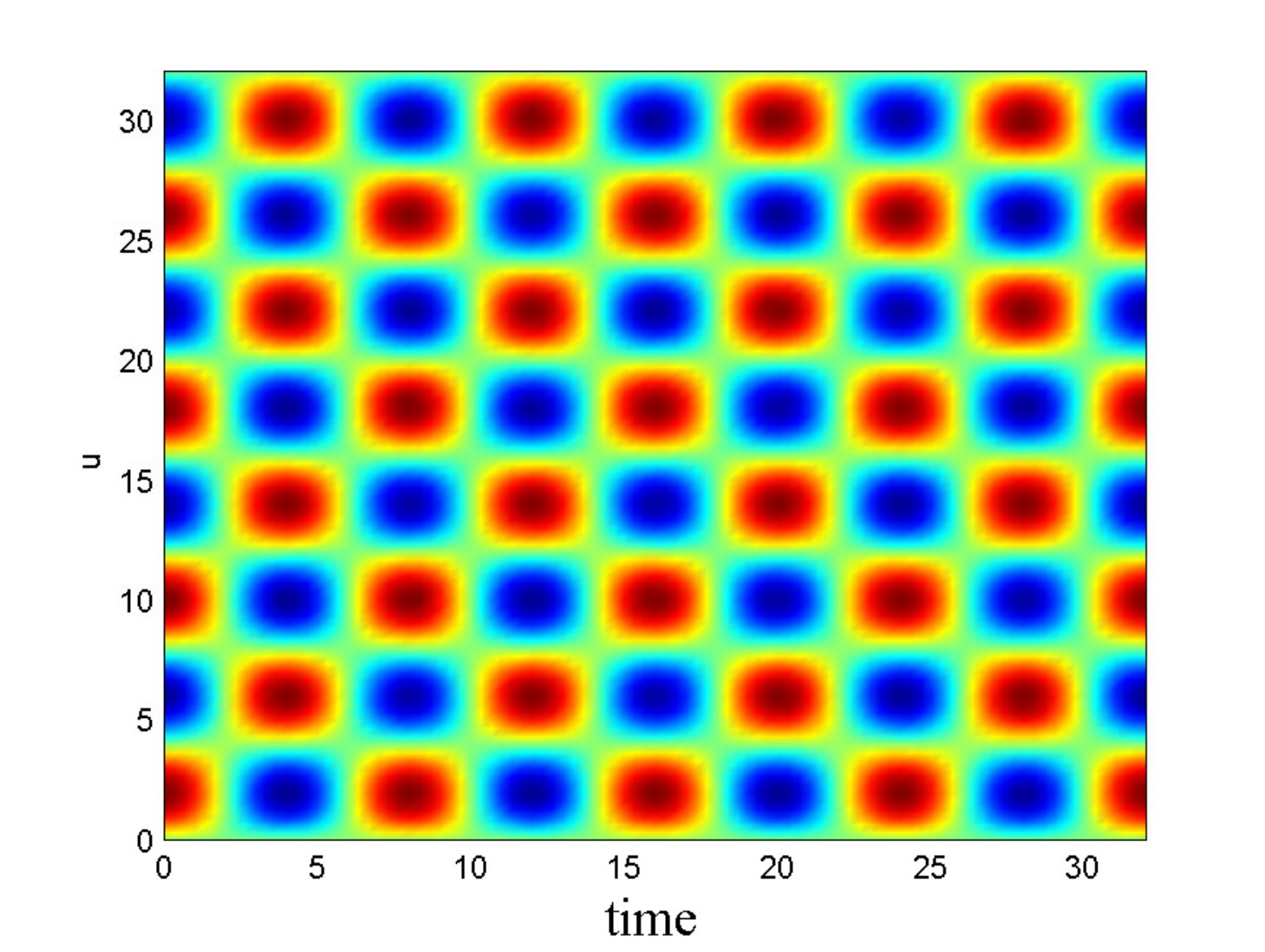}}
	\subfigure[t=128]{
		\includegraphics[width=1.4in]{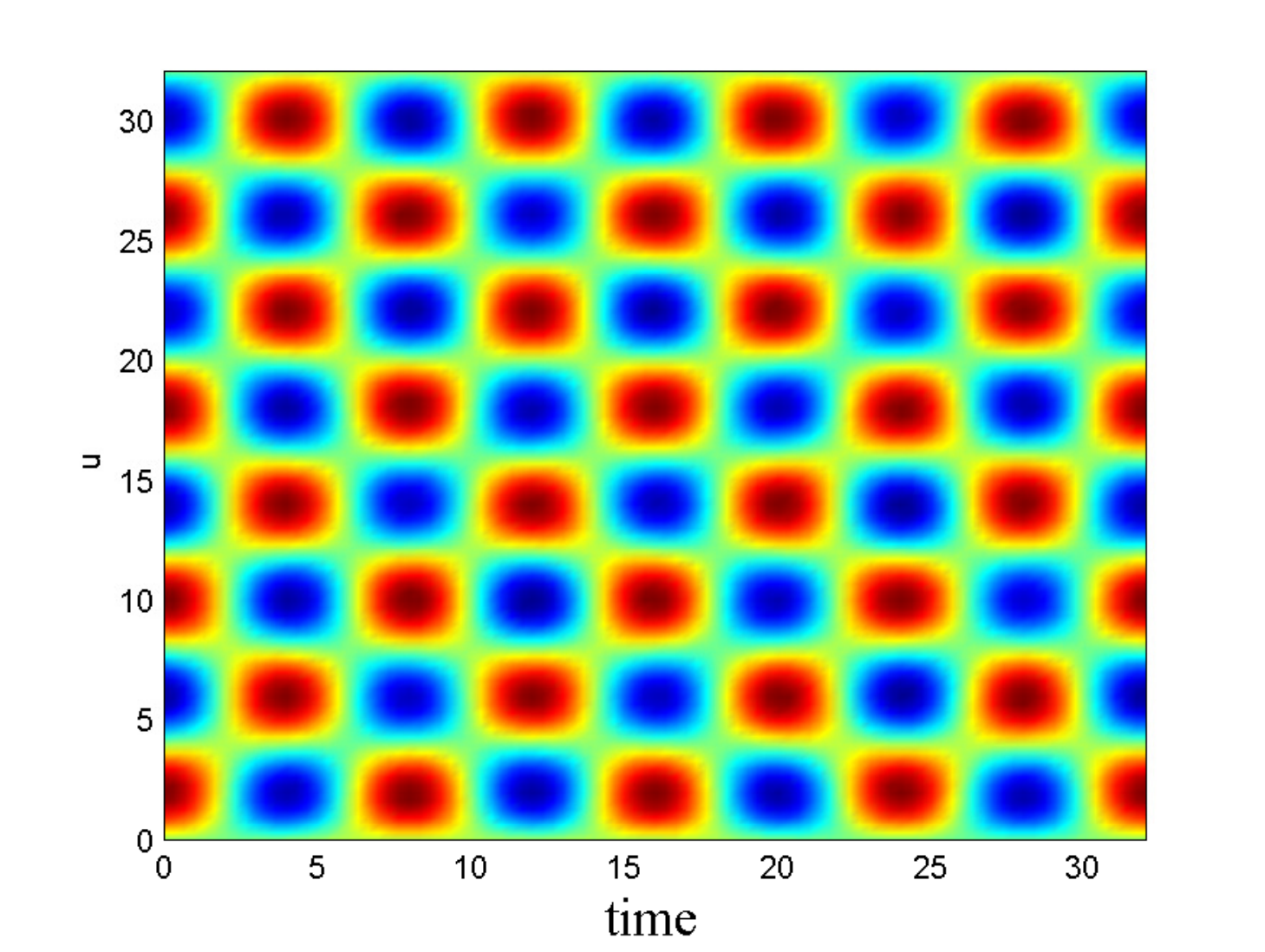}}
	\subfigure[t=256]{
		\includegraphics[width=1.4in]{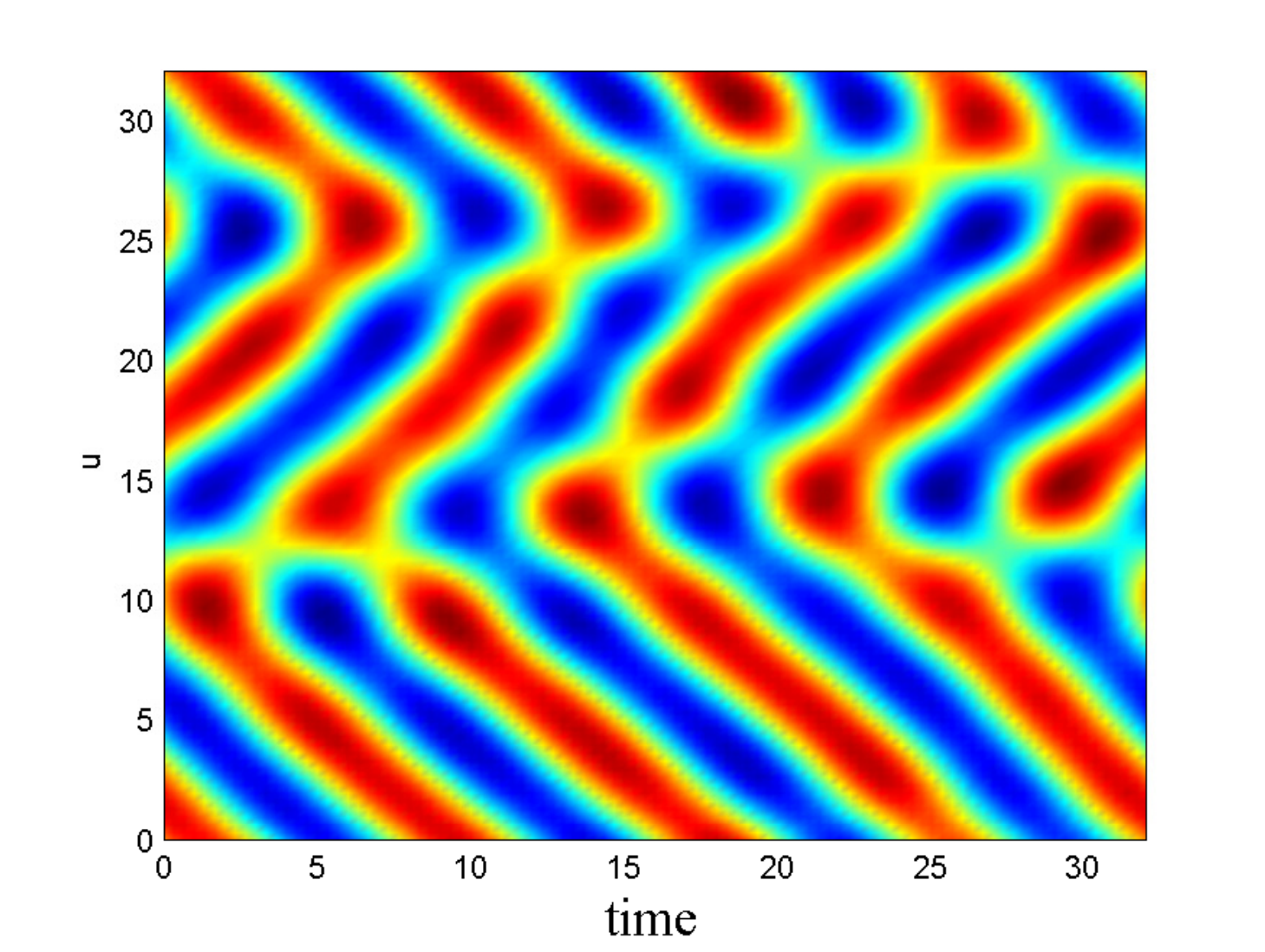}}
	\subfigure[t=512]{
		\includegraphics[width=1.4in]{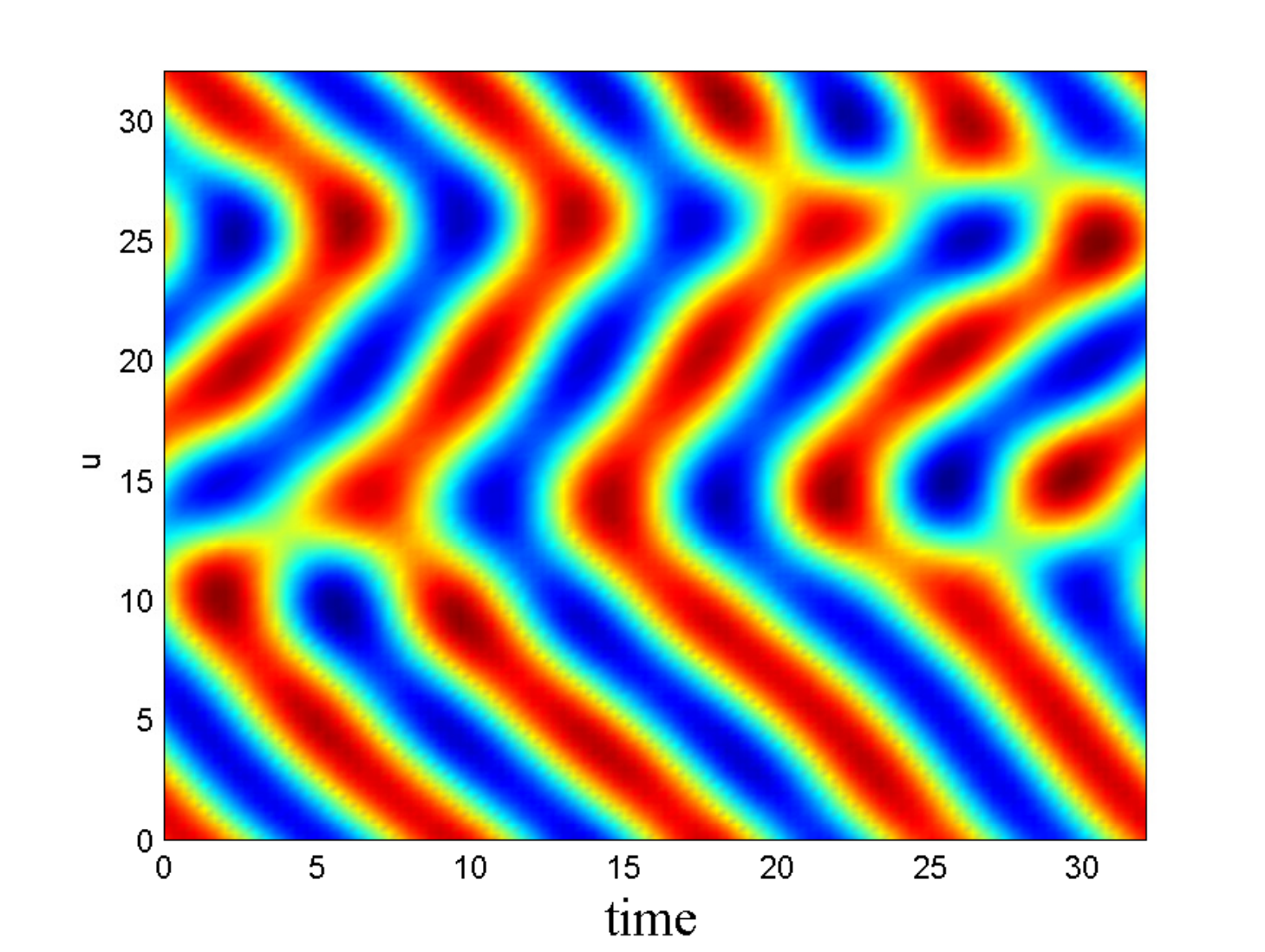}}
	\caption{Time snapshots of TFSH model at $t=64, 128, 256, 512$ with $\alpha=0.6$ and $\mathrm{g}=0$}\label{snapshot-figureg1}
	\label{fig:solution_g=0}
\end{figure}

\begin{figure}[tbh!]
	\centering
	\subfigure[energy]{
		\includegraphics[width=2in]{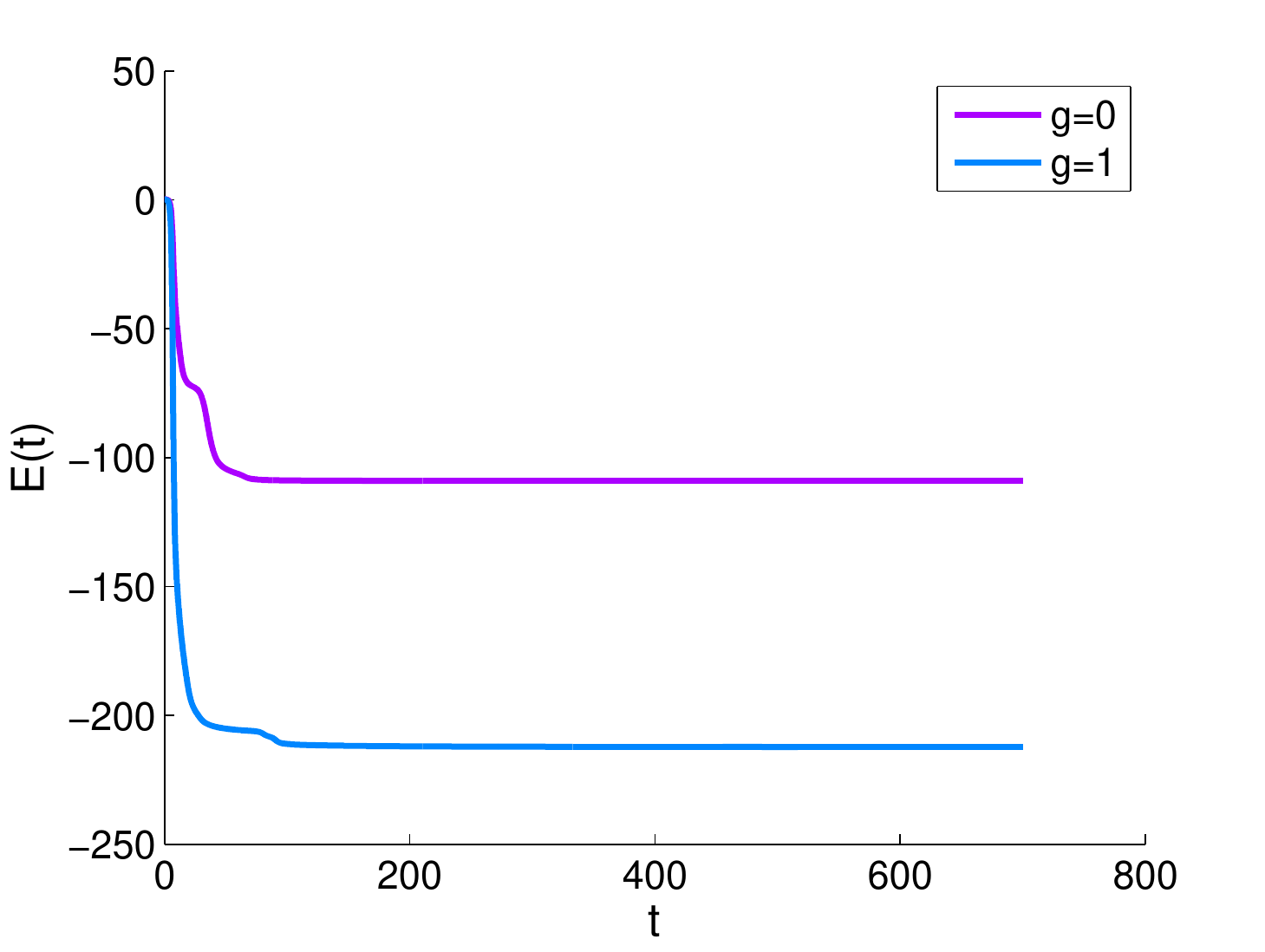}}
	\subfigure[vatiational energy]{
		\includegraphics[width=2in]{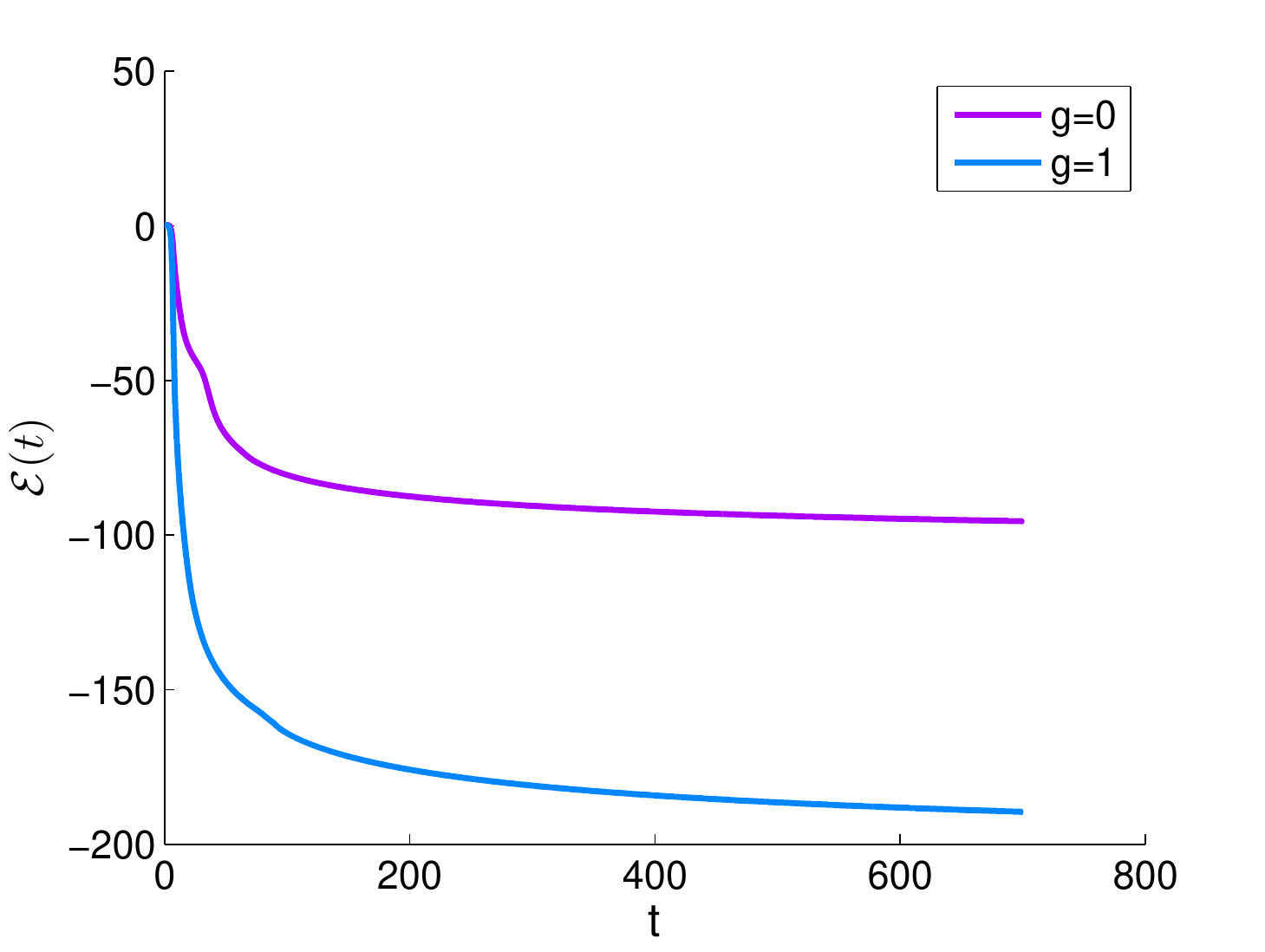}}
	\subfigure[time steps]{
		\includegraphics[width=2in]{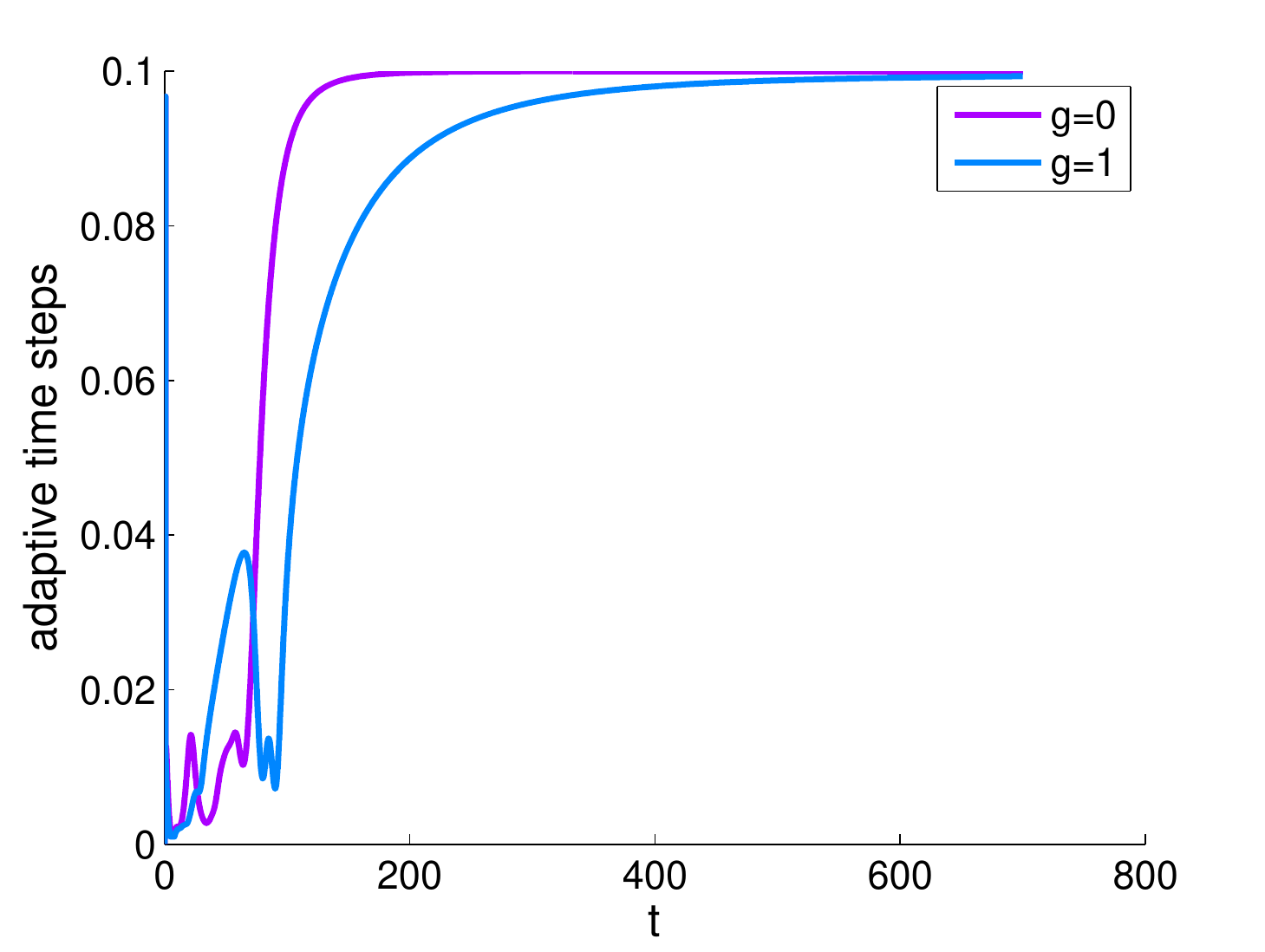}}
	\caption{Evolution of energy with $\alpha=0.6$ and different $\mathrm{g}$}
	\label{fig:energy_g}
\end{figure}

%解在不同alpha 和g条件下的图
Figure \ref{fig:soluton_alpha_g} shows that the pattern formation of the TFSH model \eqref{def:the TFSH equation} is affected by the order of fractional derivative as well, especially when $\mathrm{g}$ is small. The curves of the energy with $\mathrm{g}=0.2$ are depicted in Figure \ref{fig:energy_alpha}, which indicate that the energy decays rapidly in all cases and the speed with small fractional order is slower than that of the large one.

\begin{figure} [tbh!]
	\centering
	\includegraphics[height=240pt,width=6.0in]{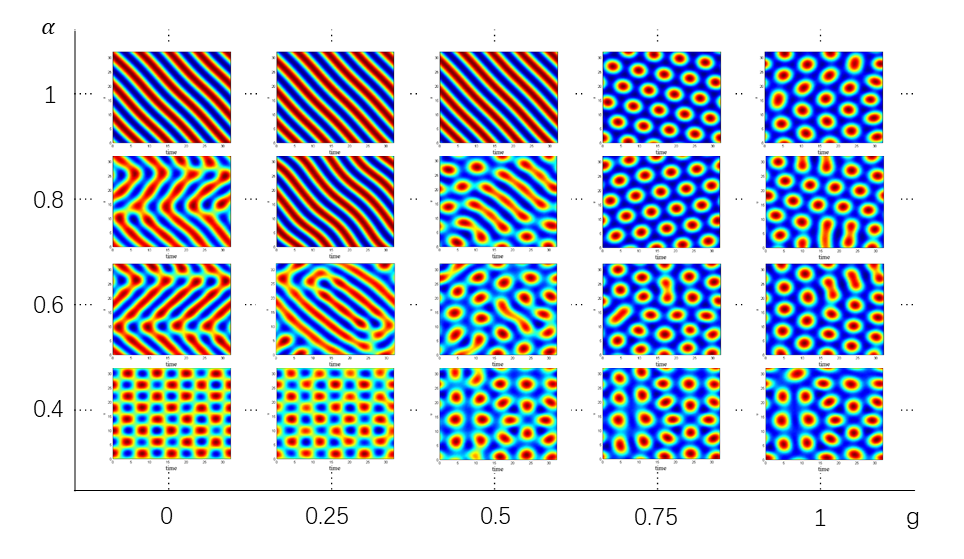}
	\caption{$u(x,y,512)$ for various $\alpha$ and $\mathrm{g}$.}
	\label{fig:soluton_alpha_g}
\end{figure}

%能量耗散图
\begin{figure}[tbh!]
	\centering
	\subfigure[energy]{
		\includegraphics[width=2in]{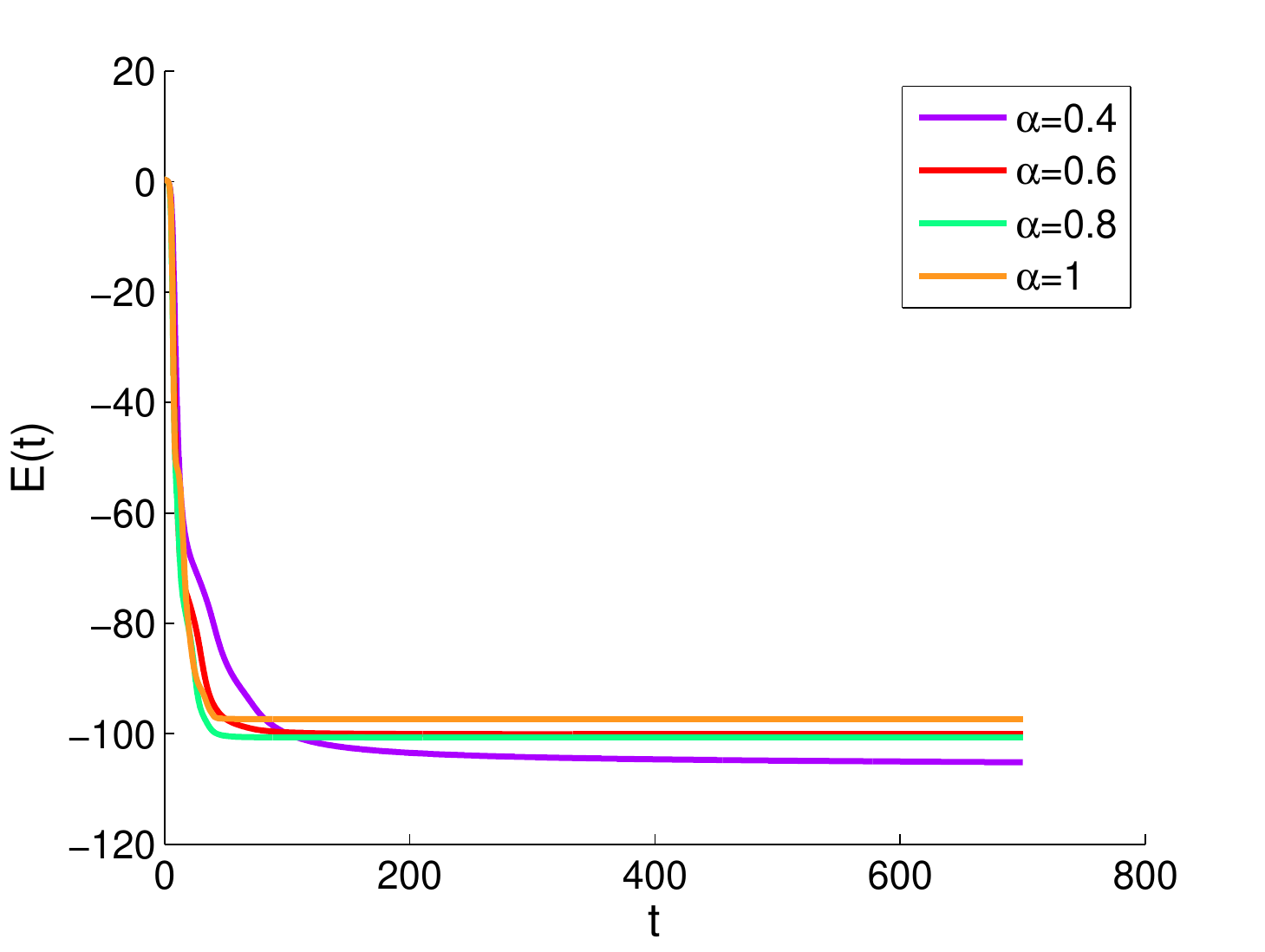}}
	\subfigure[vatiational energy]{
		\includegraphics[width=2in]{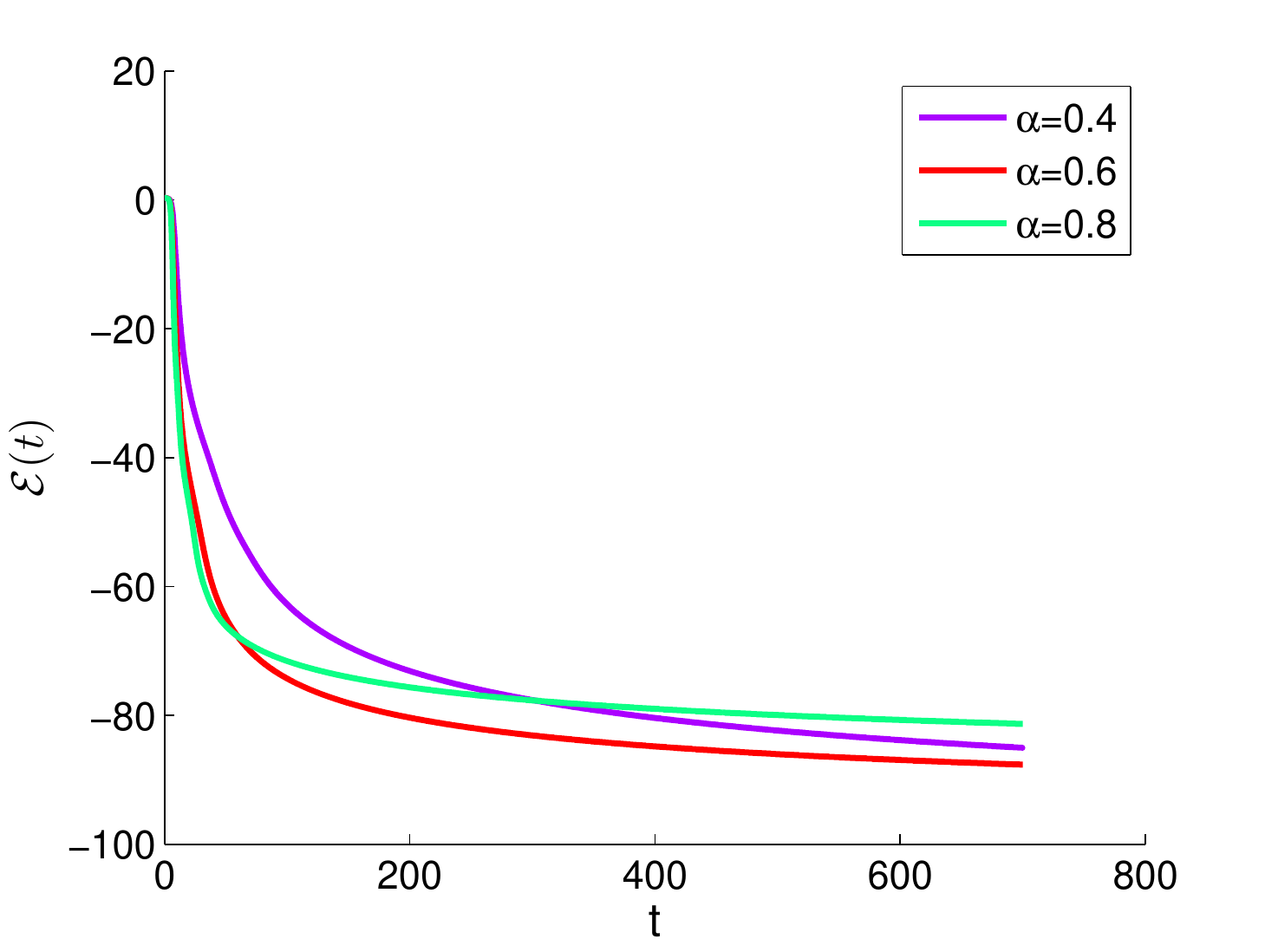}}
	\subfigure[time steps]{
		\includegraphics[width=2in]{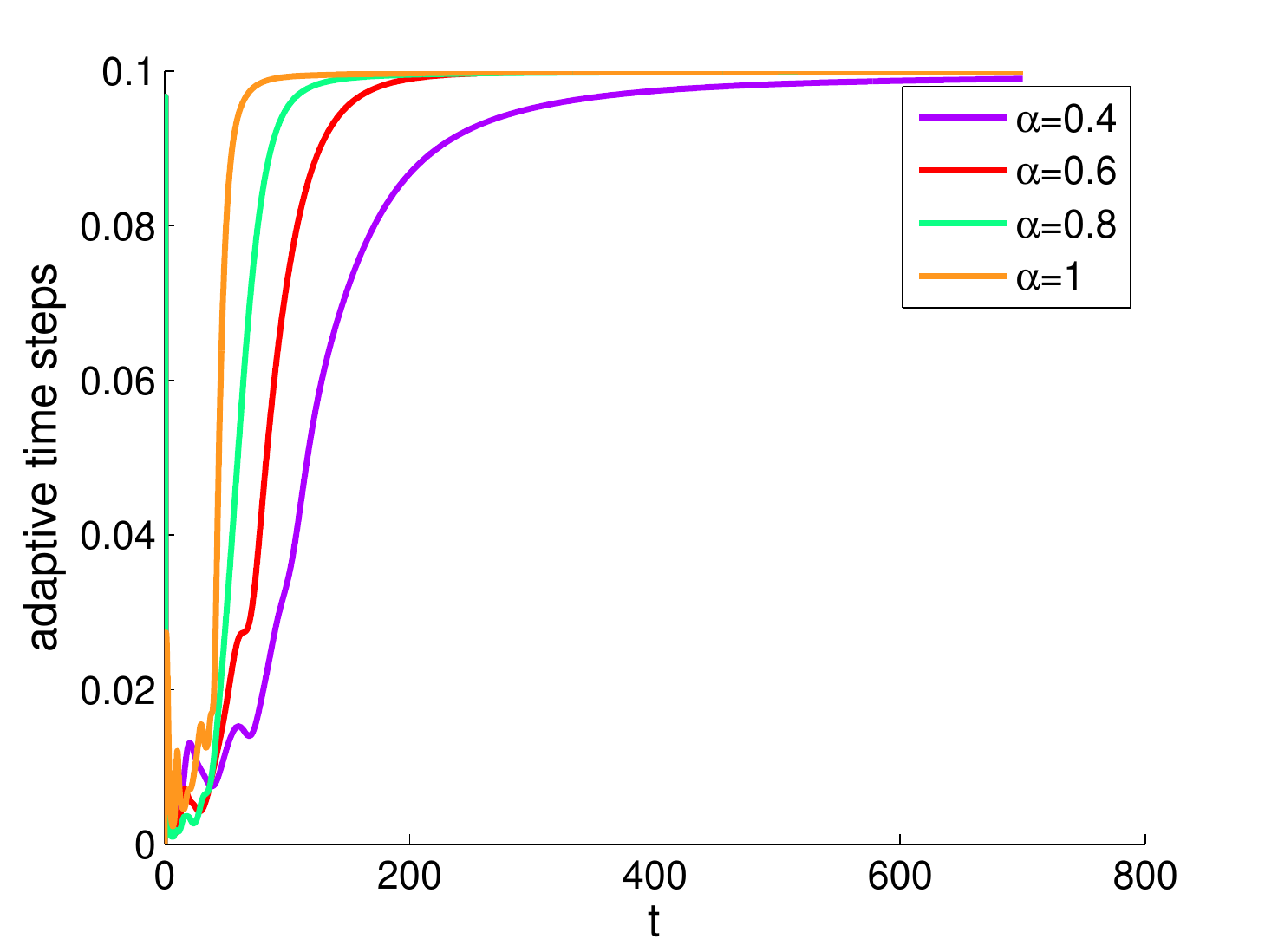}}
	
	\caption{Evolution of energy with $\mathrm{g}=0.2$ and different $\alpha$}
	\label{fig:energy_alpha}
\end{figure}
\section{Conclusions}
The variable-step L1 scheme for the TFSH equation (1.3) is derived and analyzed in this paper. By taking advantage of the discrete gradient structure of the L1 formula on the nonuniform mesh, the numerical scheme is proved to satisfy the discrete energy dissipation law. Furthermore, the $L^2$ norm error estimate of the numerical scheme is given in virtue of the global consistency error analytical technique. Based on the theoretical results, the long time simulations of the energies are shown numerically. Besides, the effects for the order of the fractional derivative and the cubic term in the energy function on the pattern formation are also presented in the computations. In the future work, high order discretizations for the time fractional derivative could be considered to establish compatible energies for the TFSH equation (1.3) or even higher order space derivative models. 
\section*{Acknowledgement}
We would like to acknowledge support by the National Natural Science Foundation of China (No. 11701081,11861060), the Fundamental Research Funds for the Central Universities, Key Project of Natural Science Foundation of China (No. 61833005) and ZhiShan Youth Scholar Program of SEU, China Postdoctoral Science Foundation (No. 2019M651634), High-level Scientific Research foundation for the introduction of talent of Nanjing Institute of Technology (No. YKL201856).

\end{document}